\documentclass[a4paper,11pt]{amsart}
\usepackage[T1]{fontenc}
\usepackage[english]{babel}
\usepackage{amsfonts, amsmath, amssymb, mathrsfs}
\usepackage{amsthm, url}
\usepackage{dsfont}
\usepackage{graphicx,epstopdf}
\usepackage{mathdots}
\usepackage{hyperref}
\usepackage{paralist}
\usepackage{soul}
\usepackage{cleveref}
\usepackage{calc}
\usepackage{tikz} 
 
\usetikzlibrary{matrix} 
\usepackage{color}
\graphicspath{{graphics/}}
\theoremstyle{plain}
\newtheorem{Theorem}{Thm}[section]
\newtheorem{Thm}[Theorem]{Theorem}
\newtheorem{Lem}[Theorem]{Lemma}
\newtheorem{Cor}[Theorem]{Corollary}
\newtheorem{Prop}[Theorem]{Proposition}

\newtheorem*{Thm*}{Theorem}
\newtheorem*{Lem*}{Lemma}
\newtheorem*{Rem*}{Remark}

\theoremstyle{definition}
\newtheorem{Def}[Theorem]{Definition}
\newtheorem{Exm}[Theorem]{Example}
\newtheorem{Rem}[Theorem]{Remark}
\newcommand{\C}{\mathbb{C}}
\newcommand{\N}{\mathbb{N}}
\renewcommand\P{\mathbb{P}}
\newcommand{\R}{\mathbb{R}}
\newcommand\scp[1]{\langle #1\rangle}
\newcommand{\cB}{\mathcal{B}}
\newcommand{\cF}{\mathcal{F}}
\newcommand{\cK}{\mathcal{K}}
\newcommand{\cV}{\mathcal{V}}
\newcommand\wh{\widehat}

\newcommand{\ii}{\operatorname{i}}
\newcommand{\id}{\mathds{1}}

\newcommand{\reg}{\mathrm{reg}}

\DeclareMathOperator\aff{aff}
\DeclareMathOperator\tr{tr}
\DeclareMathOperator\cl{clos}
\DeclareMathOperator\co{cc}
\DeclareMathOperator\cv{conv}
\DeclareMathOperator\hyper{h}

\usepackage{enumitem}
\setlist[enumerate]{%
topsep=.5ex plus0.5ex minus0.2ex,  
parsep=.5ex plus0.3ex minus0.1ex,  
itemsep=.5ex plus0.5ex minus0.2ex, 
leftmargin=4ex,    
itemindent=0ex,    
labelwidth=2ex,    
labelsep=1ex,      
listparindent=0ex, 
rightmargin=0ex,   
align=left,        
font=\rmfamily,
label=\arabic*)}
\begin{document}
\title[Kippenhahn's Theorem for Joint Numerical Ranges]{Kippenhahn's Theorem for 
Joint Numerical Ranges and Quantum States}
\author{Daniel Plaumann}
\author{Rainer Sinn}
\author{Stephan Weis}
\begin{abstract}
Kippenhahn's Theorem asserts that the numerical range of a matrix is the convex 
hull of a certain algebraic curve. Here, we show that the joint numerical range 
of finitely many Hermitian matrices is similarly the convex hull of a 
semi-algebraic set. We discuss an analogous statement regarding the dual convex 
cone to a hyperbolicity cone and prove that the class of bases of these 
dual cones is closed under linear operations. The result offers a new geometric 
method to analyze quantum states.
\end{abstract}
\date{June 19th, 2020}
%
%
%
%
%
\maketitle
%
%
\section{Introduction}
\label{sec:intro}
Let $H_d$ be the real vector space of Hermitian $d\times d$ matrices. We denote 
the set of {\em density matrices} by
\begin{equation}\label{eq:density-matrices}
\cB =\cB_d
=\{\rho\in H_d \colon \rho\succeq 0, \tr(\rho)=1\}, 
\end{equation}
where $A\succeq 0$ means that $A\in H_d$ is positive semi-definite. 
The letter $\cB$ underlines that the set $\cB$ is a base of the cone of 
positive semi-definite matrices (see \Cref{sec:duality-selfdual}). We use 
the bilinear form $\langle A,B\rangle=\tr(AB)$, $A,B\in H_d$, to identify $H_d$ 
and its dual space. Let $A_1,A_2,\dots,A_n\in H_d$, $n\in\N$, and define 
\begin{equation}\label{eq:jnr}
W = W_{A_1,A_2,\ldots,A_n} =
\{(\langle \rho, A_1\rangle,\dots,\langle \rho, A_n\rangle) \colon
\rho\in\cB_d\},
\end{equation}
a convex compact subset of the dual space $(\R^n)^\ast$ to $\R^n$. The set $W$ 
has been called {\em joint numerical range} in operator theory, see Section~2 of 
\cite{BonsallDuncan1971} (also joint algebraic numerical range \cite{Mueller2010}). 
\par
Our motivation for this paper is matrix theory and quantum mechanics. Physicists 
call density matrices {\em quantum states}, as density matrices describe the physical 
states of a quantum system
including all statistical properties \cite{Holevo2011}. 
Hence, our results contribute to the geometry of quantum states 
\cite{AubrunSzarek2017,BengtssonZyczkowski2017}. 
\Cref{sec:qm} presents numerical range methods in quantum mechanics.
\par
Perhaps more commonly, the term joint numerical range refers to 
\begin{equation}\label{eq:jnr-pure}
W^\sim = W^\sim_{A_1,A_2,\ldots,A_n} =
\{(\langle\psi|A_1\psi\rangle,\dots,\langle\psi|A_n\psi\rangle) \colon
\psi\in\C^d, \langle \psi|\psi\rangle=1\},
\end{equation}
where $\langle\varphi|\psi\rangle$ is the standard inner product of 
$\varphi,\psi\in\C^d$, see \cite{ChienNakazato2010,LiPoon2000}. 
A {\em pure state} is a projection $\rho$ onto the span of a unit vector $\psi\in\C^d$. Since $\langle\rho,A\rangle=\langle\psi|A\psi\rangle$ holds for all $A\in H_d$, the set $W^\sim$ is a linear image of the set of pure states. As the pure states are the extreme points of the set of density matrices $\cB$, the joint numerical range $W$ is the convex hull of $W^\sim$.
\par
Thus, $W=W^\sim$ holds when $W^\sim$ is convex. This is the case if 
\mbox{$n=2$} by the Toeplitz-Hausdorff theorem \cite{Hausdorff1919,Toeplitz1918}, 
whose 100th anniversary we celebrated at the time of writing. Note that 
$W_{A_1,A_2}^\sim$ is the standard {\em numerical range} 
$W(A) = \{\langle\psi|A\psi\rangle : \psi\in\C^d, \langle \psi|\psi\rangle=1\}
\subset\C$ of $A=A_1+\ii A_2$. Also, $W^\sim$ is convex if $n=3$ and 
$d\geq 3$ \cite{Au-YeungPoon1979}.
The convexity of $W^\sim$ is an open problem for $n>3$, see 
\cite{Gutkin2004,LiPoon2000}. Numerical ranges are generally nonconvex 
if the complex field is replaced with the skew field of the quaternions
\cite[p.~39]{Rodman2014}.
\par
Algebraic geometry has been employed to study numerical ranges since the 1930s. We 
consider the determinant
\[
p=\det( x_0 \id + x_1 A_1 + \cdots + x_n A_n )
\in\R[x_0,x_1,\dots,x_n],
\]
where $\id$ denotes the $d\times d$ identity matrix, and the complex projective 
hypersurface
\begin{equation*}
\cV(p)=
\{x\in\P^n \mid p(x) = 0 \}.
\end{equation*}
If $n=2$, then $\cV(p)\subset\P^2$ is an algebraic curve. Murnaghan 
\cite{Murnaghan1932} showed that the eigenvalues of the matrix $A_1+\ii A_2$ are the 
foci of the curve
\begin{equation*}
T=\{ y_1+\ii y_2 \mid y_1,y_2\in\R,(1:y_1:y_2)\in \cV(p)^\ast\}\subset\C,
\end{equation*}
where $X^\ast\subset(\mathbb{P}^n)^\ast$ denotes the dual variety parametrizing 
hyperplanes tangent to a variety $X\subset\mathbb{P}^n$ (cf.~\Cref{sec:duality-cones} 
for more details). Kippenhahn recognized the meaning of the convex hull of the curve 
$T$.
\par
\begin{Thm}[Kippenhahn \cite{Kippenhahn1951}]\label{thm:k}
The numerical range of $A_1+\ii A_2$ is the convex hull of the curve $T$, 
in other words, $W(A_1+\ii A_2)=\cv(T)$.
\end{Thm}
\Cref{thm:k} is a well-known tool in matrix analysis 
\cite{Bebiano-etal2005,Fiedler1981,GallaySerre2012,JoswigStraub1998}. The curve 
$T\subset\C$ is called the {\em Kippenhahn curve} or the {\em boundary generating curve} 
of the numerical range $W(A_1+\ii A_2)$. The curve $T$ has proven useful in 
classifications of numerical ranges of $3$-by-$3$ matrices 
\cite{Keeler-etal1997,Kippenhahn1951} and $4$-by-$4$ matrices 
\cite{Camenga-etal2019,ChienNakazato2012}, and it has been studied for special 
matrices \cite{ChienNakazato2018,GauWu2004,Gau-etal2013,Militzer-etal2017}.
\par 
We sketch a proof of \Cref{thm:k} in a manner that may help to explain the 
geometry behind the proof of \Cref{thm:dual-hyp-cone} later on: Consider the 
convex set $S=\{(x_1,x_2)^T\in\R^2\mid \id+ x_1 A_1+x_2 A_2\succeq 0\}$ (a 
{\em spectrahedron}). Assume for simplicity that $S$ is compact, $X=\cV(p)$ is 
smooth, and that the degree $d$ of $p$ is even. The curve $X$ is hyperbolic, i.e.~its real 
points consist of $\frac d2$ nested ovals in the real projective plane. The 
innermost oval is the boundary of $S$. All but finitely many points of the dual 
curve $X^\ast$ correspond to simple tangent lines to $X$. The set of real points of the 
dual curve $X^\ast$ (of degree $d(d-1)$) again consists of $\frac d2$ nested connected components, 
together with at most finitely many isolated real (singular) points. The tangent 
lines to the boundary of $S$ now correspond to the outermost oval of $X^\ast$, since all other tangent lines to $X$ do not pass through $S$ (see \Cref{Cor:multiplicity}). 
The outermost oval therefore bounds the convex dual $S^\circ$ of $S$, which is exactly 
the numerical range $W(A_1+\ii A_2)$. The claim of Kippenhahn's theorem follows 
if we can show that none of the isolated real singularities of $X^\ast$ lie outside 
of $W$; see \Cref{thm:dim3}.
\par
Chien and Nakazato \cite{ChienNakazato2010} provided a more rigorous proof of 
\Cref{thm:k} compared to Kippenhahn's. They also found a triple of Hermitian 
$3\times 3$-matrices for which the literal analogue of \Cref{thm:k} fails in dimension 
$n=3$. We will see that the last part of the above sketch in the plane case, the 
position of singular points of the dual curve, is exactly what causes the failure of 
the theorem in higher dimensions. This can also be seen in the counterexample by Chien 
and Nakazato, \Cref{ex:ChienNakazato}. 
\par
By removing all singular points from the projective dual variety $\cV(p)^\ast$ and taking 
Euclidean closure, we will prove that a modified version of \Cref{thm:k} is valid in all dimensions.
Let $X_1,\dots,X_r$ denote the irreducible components of the hypersurface $\cV(p)$. 
We consider the set $(X_i^\ast)_\reg$ of the regular points of the dual variety 
$X_i^\ast$, the set
\begin{equation*}
T_i = \left\{ (y_1,\dots,y_n) \in (\R^n)^\ast \mid
(1:y_1:\dots:y_n)\in (X_i^\ast)_\reg \right\},
\qquad
i=1,\dots,r, 
\end{equation*}
and the Euclidean closure $T^\sim=\cl(T_1\cup\cdots\cup T_r)$ of the union 
$T_1\cup\cdots\cup T_r$. Our main result is as follows.
\par
\begin{Thm}\label{eq:modified-K}
The joint numerical range $W$ is the convex hull of $T^\sim$.
\end{Thm}
\begin{Rem}\label{rem:main-theorem}
\begin{enumerate}
\item
\Cref{eq:modified-K} implies that the set $T^\sim$ contains all extreme points of the 
compact, convex set $W$. We show in the proof of \Cref{thm:dual-hyp-cone} that
$T^\sim$ contains the exposed points of $W$, and hence all of the extreme points by a limiting argument (Straszewicz's Theorem). But just as in Kippenhahn's original theorem, $T^\sim$ is not necessarily contained in the boundary of $W$\!, only in $W$. 
\item
We point out that \Cref{thm:dual-hyp-cone}, which implies the theorem stated here, holds more generally for hyperbolic hypersurfaces rather than just determinantal hypersurfaces. While this makes no difference in the plane, by the Helton-Vinnikov theorem, the statement is indeed more general in higher dimensions, and the proof relies purely on the real geometry of hyperbolic polynomials.
\item
The joint numerical range $W$ is a semi-algebraic set as it is a linear image of the 
semi-algebraic set $\cB$ by quantifier elimination (see e.g.~\cite[Thm.~2.2.1]{BCR}). 
The set $T_1\cup\cdots\cup T_r$ and hence its Euclidean closure $T^\sim$ are 
semi-algebraic sets as well.
\item
\Cref{eq:modified-K} not only describes a semi-algebraic set that contains the extreme points of $W$ but more precisely the Zariski closure of the set of extreme points: The union of the dual varieties $X_i^\ast$ of the irreducible components $X_i$ of the algebraic boundary of the hyperbolicity cone of $p$ is the Zariski-closure of the set of extreme points of $W$, see \Cref{rem:zariski-extreme} for details.
\end{enumerate}
\end{Rem}
We organize the article as follows. 
\Cref{sec:prelim} collects preliminaries
from convex geometry and real algebraic geometry.
\Cref{sec:duality-cones} presents 
a detailed discussion of the remarkable fact, proved by the second author in 
\cite{Sinn2015}, that the dual convex cone $C^\vee$ to a hyperbolicity cone $C$ 
is the closed convex cone generated by a particular semi-algebraic set. This 
implies that every base of $C^\vee$ is the closed {\em convex hull} of a section 
of that semi-algebraic set. The same is true for linear images of the bases as we 
show in \Cref{sec:duality-selfdual}, because (up to 
isomorphism) they are bases of dual convex cones to sections of $C$, which are 
hyperbolicity cones themselves. Returning to the cone of positive semi-definite 
matrices in \Cref{sec:PSD-cone}, we obtain a proof of 
\Cref{eq:modified-K} and discuss examples. We analyze the case $n=2$ 
separately in \Cref{thm:dim3}, which yields a proof of Kippenhahn's original 
result as stated in \Cref{thm:k}.
\par
%
%
%
\section{Connections to Quantum Mechanics}
\label{sec:qm}
Physicists refer to linear images of certain subsets of the set of quantum states 
$\cB_d$ as \emph{numerical ranges}. Often they consider images under a map 
$\cB_d\to\R^n$, $\rho\mapsto\langle \rho, A_i\rangle_{i=1}^n$, where 
$A_1,A_2,\dots,A_n\in H_d$ are Hermitian matrices. We discuss examples where 
algebraic geometry could help solving problems of quantum mechanics in the 
context of numerical ranges. In this paper, we ignore numerical ranges outside the 
pattern of linear images of subsets of $\cB_d$, for example higher-rank numerical 
ranges \cite{Choi-etal2006}.
\par
%
%
\subsection{Linear images of the set of all quantum states}
\label{sec:qm-all}
The joint numeri\-cal range $W_{A_1,A_2,\dots,A_n}$, as defined in \Cref{eq:jnr}, 
appears in problems of experimental and theoretical physics.
\par
The geometry of the joint numerical range has been of direct interest to the 
experimentalists Xie et al.\ \cite{Xie-etal2019}. Their drawings of data from 
photonic experiments show ellipses on the boundary of the joint numerical range 
$W_{A_1,A_2,A_3}$ of three $3$-by-$3$ matrices $A_1,A_2,A_3\in H_3$, clearly in 
agreement with the classification: the exposed faces of positive dimensions are 
ellipses and segments assembling one of ten configurations \cite{Szymanski-etal2018}. 
The possibility to carry out experiments with $4$-level quantum systems calls for a 
similar classification of $W_{A_1,A_2,A_3}$ for $4$-by-$4$ matrices 
$A_1,A_2,A_3\in H_4$.
\par
Using \Cref{thm:k} and analyzing the Kippenhahn curve, Kippenhahn \cite{Kippenhahn1951} 
obtained a classification of the numerical range $W(A_1+\ii A_2)=W_{A_1,A_2}$ in terms 
of flat portions on the boundary of $W_{A_1,A_2}$ for all Hermitian \mbox{$3$-by-$3$} 
matrices $A_1,A_2\in H_3$, see also \cite{Keeler-etal1997}. A similar approach has been 
taken for $4$-by-$4$ matrices
\cite{Camenga-etal2019,ChienNakazato2012}. We are invited to describe the exposed 
faces of positive dimensions of the joint numerical range $W_{A_1,A_2,A_3}$, employing 
\Cref{eq:modified-K} and studying the semi-algebraic set $T^\sim$, whose convex hull is
$W_{A_1,A_2,A_3}$. This should reproduce the classification of the set $W_{A_1,A_2,A_3}$ 
for $3$-by-$3$ matrices \cite{Szymanski-etal2018} and lead to a classification for 
$4$-by-$4$ matrices.
\par
Quantum thermodynamics \cite{YungerHalpern-etal2016} describes equilibrium states 
with multiple conserved quantities $A_1,A_2,\dots,A_n$ in terms of 
\emph{generalized Gibbs states}
\[
\rho_x=\tfrac{e^{x_1A_1+\dots+x_nA_n}}{\tr e^{x_1A_1+\cdots+x_nA_n}}\in\cB_d,
\qquad x=(x_1,\dots,x_n)^T\in\R^n.
\]
The ``boundary at infinity'' ($|x|\to\infty$) to the manifold $\{\rho_x : x\in\R^n\}$  
may not be closed in the Euclidean topology if the matrices $A_1,A_2,\dots,A_n$ fail 
to commute \cite{WeisKnauf2012}. The physical meaning of this topological problem 
remains mysterious \cite{Chen-etal2015}. Mathematically, the discontinuity depends on 
the geometry of the joint numerical range $W_{A_1,A_2,\dots,A_n}=\cv(T^\sim)$ 
\cite{Chen-etal2015,Rodman-etal2016,SpitkovskyWeis2018}. Hence, the semi-algebraic set 
$T^\sim$ explains infinitesimal properties of the manifold of generalized Gibbs states 
near the boundary at infinity.
\par
A connection between functional analysis and algebraic geometry awaits further 
investigation. The \emph{Wigner distribution} of a quantum state $\rho\in\cB_d$ with 
respect to $A_1,A_2,\dots,A_n$ is the tempered distribution ${\mathcal W}_\rho$ on 
$\R^n$ that satisfies
\[\textstyle
\int \operatorname{d}\!a \,{\mathcal W}_\rho(a_1,\dots,a_n) 
f(x_1 a_1 + \cdots + x_n a_n)
= \langle\rho,f(x_1 A_1 + \cdots + x_n A_n)\rangle
\]
for all $x\in\R^n$ and all infinitely differentiable functions $f:\R\to\C$. The Wigner 
distribution is a common tool in quantum optics and theoretical physics. Schwonnek and 
Werner \cite{SchwonnekWerner2018} showed that the distribution ${\mathcal W}_\rho$ is 
compactly supported on the joint numerical range $W_{A_1,A_2,\dots,A_n}$ and that the 
singularities of ${\mathcal W}_\rho$ lie in the semi-algebraic set $T^\sim$.
\par
%
%
\subsection{Linear images of subsets of the set of quantum states}
\label{sec:qm-subsets}
Linear images of semi-algebraic subsets of $\cB_d$ are amenable to algebraic geometry
as well. We mention examples relevant to quantum mechanics.
\par
Algebraic geometry \cite{Conca-etal2015} has proven helpful in the theory of pure 
state tomography \cite{Heinosaari-etal2013}, the reconstruction of a pure state 
$\rho\in\cB_d$ from its expected value tuple $\langle\rho,A_i\rangle_{i=1}^n\in W^\sim$ 
in the linear image $W^\sim$ of the set of pure states defined in \Cref{eq:jnr-pure}. 
A different topic, for example, in density functional theory \cite{Chen-etal2012}, 
is describing the extreme points of the joint numerical range $W$. Both $W^\sim$ and 
$T^\sim$ are semi-algebraic subsets of $W$ that contain the extreme points of $W$. 
The set $T^\sim$ is especially suitable to study the extreme points of $W$, 
see \Cref{rem:zariski-extreme}.
\par
Many-particle systems are fascinating due to interaction and correlation between the
units. The simplest example in the quantum domain is the two-qubit system with
Hilbert space $\C^4=\C^2\otimes\C^2$. Physicists 
\cite{Chen-etal2016,Gawron-etal2010} have studied the 
\emph{joint product numerical range} of $A_1,A_2,\dots,A_n\in H_4$,
\[
\Pi
=\{\langle\psi\otimes\varphi|A_i\psi\otimes\varphi\rangle_{i=1}^n \mid
\psi,\varphi\in\C^2, \langle \psi|\psi\rangle=\langle \varphi|\varphi\rangle=1\},
\]
a linear image of the set of product states $\sigma\otimes\tau$, where 
$\sigma,\tau\in\cB_2$ are pure states. The convex hull of the product states is the 
set of {\em separable} states, the states that lack the genuine quantum correlation 
called {\em entanglement} \cite{AubrunSzarek2017,BengtssonZyczkowski2017}. Hence, 
the set $\Pi$ and its convex hull allow us to study quantum correlations.
\par
Two-qubit density matrices offer insights into statistical mechanics. As per 
the {\em quantum de Finetti theorem} \cite{Lewin-etal2014,Stormer1969}, the 
two-particle marginals of an infinite bosonic qubit-system are convex combinations 
of symmetric product states $\sigma\otimes\sigma$, where $\sigma\in\cB_2$ is a 
pure state. The ground state energy of an energy operator with two-party 
interactions on an infinite bosonic qubit-system is the distance of the origin from 
a supporting hyperplane to the set
\[
\Pi^\mathrm{sym}_{A_1,A_2,\dots,A_n} =  
\{\langle\psi\otimes\psi|A_i\psi\otimes\psi\rangle_{i=1}^n \mid
\psi\in\C^2, \langle \psi|\psi\rangle=1\}
\]
for suitable matrices $A_1,A_2,\dots,A_n\in H_4$. Notably, ruled surfaces on the 
boundary of the convex hull of $\Pi^\mathrm{sym}_{A_1,A_2,A_3}\subset\R^3$ are 
expressions of phase transitions \cite{Chen-etal2016,Zauner-etal2016}. An analogue 
to \Cref{eq:modified-K} regarding the set $\Pi^\mathrm{sym}_{A_1,A_2,A_3}$ 
would be helpful for the analysis of bosonic qubit-systems.
\par
%
%
\section{Preliminaries}
\label{sec:prelim}
We collect terms, basic results, and references to the literature regarding 
convex geometry and (real) algebraic geometry.
\par
%
%
\subsection{Convex Geometry}
\label{sec:conv-geo}
We discuss various notions of cones in a finite-dimensional real vector space $V$: 
cones (which may be nonconvex), convex cones, and normal cones. Additionally, we 
define affine cones over complex projective varieties in \Cref{sec:rag}. As a 
general reference for convex geometry, we recommend 
\cite{Barvinok2002,Rockafellar1970}. 
\par
A subset $C$ of $V$ is a \emph{cone} if $\lambda x\in C$ whenever $\lambda>0$ and 
$x\in C$. A subset $K$ of $V$ is a \emph{convex cone} if $K$ is a nonempty convex 
set and if $\lambda x\in K$ whenever $\lambda\geq 0$ and $x\in K$. The 
\emph{affine hull} $\aff(S)$, \emph{convex hull} $\cv(S)$, \emph{cone}, 
\emph{convex cone} $\co(S)$, and \emph{closed convex cone} generated by a subset 
$S\subset V$ is the smallest affine space, convex set, cone, convex cone, and closed 
convex cone, respectively, containing $S$. 
\par
A subset $B\subset V$ is a \emph{base} of a cone $C\subset V$ if $B$ is the intersection of $C$ with an affine hyperplane, $0\not\in\aff(B)$, 
and for all nonzero points $x\in C$ there exists $y\in B$ and $\lambda>0$ such that $x=\lambda y$. 
A convex cone $K\subset V$ is \emph{pointed} if $K\cap (-K) = \{0\}$, i.e.~if $K$ 
contains no lines.
\par
We denote the dual vector space of $V$ by $V^\ast$. The \emph{annihilator} of a 
subset $S\subset V$ is
\[
S^\perp=\{\ell\in V^\ast \colon 
\ell(x)=0\,\text{ for all } x\in S\},
\]
the \emph{dual convex cone} to $S$ is
\[
S^\vee = \{\ell\in V^\ast \colon 
\ell(x)\geq 0\,\text{ for all }x\in S \},
\]
and the \emph{dual convex set} to $S$ is 
\[
S^\circ = \{\ell\in V^\ast \colon 
1+\ell(x)\geq 0\,\text{ for all } x\in S \}.
\]
We denote intersections of $S$ with affine hyperplanes avoiding the origin by
\begin{equation}\label{eq:base}
\hyper_\ell(S)=\{x\in S\colon \ell(x)=1\},
\qquad
\ell\in V^\ast, \ell\neq 0.
\end{equation}
\par
%
%
\begin{Lem}\label{lem:base-affine}
Let $C\subset V$ be a cone and let $B\neq\emptyset$ be a base of $C$. Then there exists 
a nonzero functional $\ell\in V^\ast$ such that $B=\hyper_\ell(C)$. If $C$ admits a 
compact base, then $C\cup\{0\}$ is closed. 
\end{Lem}
\begin{proof}
Let $B$ be a nonempty base of the cone $C$ and let $X$ be the linear span of $C$. 
It follows from the definition of a base that there is a nonzero linear functional 
$\hat{\ell}\in X^\ast$ such that $B\subset\hyper_{\hat{\ell}}(X)$, hence 
$B=\hyper_{\hat{\ell}}(C)$. Identifying the dual space $X^\ast$ 
with any subspace complementary to the annihilator $X^\perp$ in $V^\ast$ and extending 
$\hat{\ell}$ to a functional $\ell\in V^\ast$ by the Hahn-Banach theorem, we obtain 
$B=\hyper_\ell(C)$. 
\par
Let $(x_i)\subset C$ be a sequence converging to a nonzero point $x\in V$. If the base 
$\hyper_\ell(C)$ of $C$ is compact, then the sequence 
$(\tfrac{x_i}{\ell(x_i)})\subset\hyper_\ell(C)$ has a converging subsequence with limit 
$y\in\hyper_\ell(C)$. As $\ell(x)=\lim_{i\to\infty}\ell(x_i)\geq 0$ and as 
$x=\lim_{i\to\infty}\ell(x_i)\tfrac{x_i}{\ell(x_i)}=\ell(x)y$, 
the point $x$ lies in $C$.
\end{proof}
The biduality theorem for closed convex cones follows from the separation theorem 
in convex geometry, see for example \cite[Theorem 14.1]{Rockafellar1970}. 
\begin{Thm}\label{thm:biduality-cone}
Let $K\subset V$ be a closed convex cone. Then $(K^\vee)^\vee = K$.
\end{Thm}
We describe the family of bases of a closed convex cone. 
\par
\begin{Lem}\label{lem:char-compact-base}
Let $K\subset V$ be a closed convex cone. The following assertions 
are equivalent for all nonzero points $x\in V$.
\begin{enumerate}
\item The point $x$ is an interior point of $K$.
\item The set $\hyper_x(K^\vee)$ is a base of $K^\vee$.
\end{enumerate}
If one of these equivalent assertions is true, then the set $\hyper_x(K^\vee)$ is 
compact.
\end{Lem}
\begin{proof}
Let $x\in V$ be nonzero. By Theorem 13.1 in \cite{Rockafellar1970} the point 
$x$ is an interior point of $K$ if and only if $\ell(x)<\delta^\ast(\ell|K)$ 
holds for all $\ell\in V^\ast\setminus\{0\}$, where
\[
\delta^\ast(\ell|K)
=\sup_{x\in K}\ell(x)
=\left\{\begin{array}{ll}
0 & \text{if $\ell\in-K^\vee$,}\\
\infty & \text{if $\ell\not\in-K^\vee$,}
\end{array}\right.
\qquad\ell\in V^\ast,
\]
is the support function of $K$. Hence the assertion (1) is equivalent to $\ell(x)>0$ 
for all $\ell\in K^\vee\setminus\{0\}$, which is equivalent to $\hyper_x(K^\vee)$ 
being a base of $K^\vee$.
\par 
The compactness follows from properties of the \emph{recession cone}
$0^+(\hyper_x(K^\vee))$, the set of vectors $\ell'\in V^\ast$ such that 
$\ell+\lambda\ell'$ lies in $\hyper_x(K^\vee)$ for all $\ell\in\hyper_x(K^\vee)$ and 
$\lambda\geq 0$. We can assume that the set $\hyper_x(K^\vee)$ is nonempty. In this 
case we have $0^+(\hyper_x(K^\vee))=x^\perp\cap K^\vee$ by Coro.~8.3.2 and~8.3.3 of
\cite{Rockafellar1970} as the recession cone of $\hyper_x(V^\ast)$ is $x^\perp$.
Since $\hyper_x(K^\vee)$ is a base of $K^\vee$, we have $\ell(x)>0$ for all 
$\ell\in K^\vee\setminus\{0\}$, hence $x^\perp\cap K^\vee=\{0\}$. Now 
\cite[Thm.~8.4]{Rockafellar1970} shows that $\hyper_x(K^\vee)$ is bounded, hence
compact.
\end{proof}
Cones that are contained in a pointed closed convex cone behave nicely.
\par
\begin{Lem}\label{lem:cn-ccc}
Let $K\subset V$ be a closed convex cone with nonempty interior. Let $C\subset K^\vee$ 
be a cone and let $x$ be a nonzero interior point of $K$. Then 
\begin{equation}\label{eq:compact-base}
\hyper_x(\co(\cl(C)))
=\cv(\hyper_x(\cl(C)))
=\cv(\cl(\hyper_x(C)))
\end{equation}
is a compact base of the pointed, closed convex cone $\co(\cl(C))=\cl(\co(C))$.
\end{Lem}
\begin{proof}
Let $x\neq 0$ be an interior point of $K$. Then $\hyper_x(K^\vee)$ is a compact base 
of $K^\vee$ by \Cref{lem:char-compact-base}. \emph{A fortiori}, $\hyper_x(C)$ is a 
base of the cone $C\subset K^\vee$; thus $\hyper_x(\co(C))\subset\cv(\hyper_x(C))$. 
The converse inclusion is clear and proves the first equality sign in 
\Cref{eq:compact-base} after replacing $C$ with $\cl(C)$. The inclusion 
$\hyper_x(\cl(C))\subset\cl(\hyper_x(C))$ holds, again as $\hyper_x(C)$ is a base of 
$C$. The converse inclusion is clear and proves the second equality sign in
\Cref{eq:compact-base}. As $\hyper_x(\cl(C))$ is compact, its convex hull is compact 
by Theorem~17.2 in \cite{Rockafellar1970}. Hence $\hyper_x(\co(\cl(C)))$ is a compact 
base of the convex cone $\co(\cl(C))$, which is closed by \Cref{lem:base-affine}. 
This proves $\cl(\co(C))\subset\co(\cl(C))$; the converse inclusion is clear. 
\end{proof}
\Cref{pro:dual-of-intersection} allows us to focus on irreducible varieties in 
\Cref{sec:duality-cones}.
\par
\begin{Prop}\label{pro:dual-of-intersection}
Let $K_1,\ldots,K_r\subset V$  be convex cones and let $e\neq 0$ be an interior point 
of $K=K_1\cap K_2\cap\dots\cap K_r$. Let $C_i\subset V^\ast$ be a cone such that 
$K_i^\vee=\cl(\co(C_i))$, that is to say, the dual convex cone to $K_i$ 
is the closed convex cone generated by $C_i$ for $i=1,\ldots,r$. Then
\begin{align}\label{eq:inter-dual}
K^\vee 
= \co\big(\cl(C_1\cup C_2\cup\dots\cup C_r)\big)
\end{align}
and
\begin{align}\label{eq:inter-dual-cl}
\hyper_e(K^\vee)
= \cv\Big(\cl\big(\hyper_e(C_1)\cup\hyper_e(C_2)\cup\dots\cup\hyper_e(C_r)\big)\Big).
\end{align}
\end{Prop}
\begin{proof}
Corollary~16.4.2 of \cite{Rockafellar1970} shows 
$K^\vee=K_1^\vee+K_2^\vee+\cdots+K_r^\vee$ as $e$ is an interior point of $K_i$ for all 
$i=1,\ldots,r$, hence $K^\vee=\co(K_1^\vee\cup K_2^\vee\cup\cdots\cup K_r^\vee)$.
This proves \Cref{eq:inter-dual}, as we have, by assumption and by \Cref{lem:cn-ccc}, 
\[
K_i=\cl(\co(C_i))=\co(\cl(C_i)),
\qquad i=1,\ldots,r.
\]
Equation \eqref{eq:inter-dual-cl} follows from the equation 
$\hyper_e(\co(\cl(C)))=\cv(\cl(\hyper_e(C)))$ provided in \Cref{lem:cn-ccc}
using $C=\bigcup_{i=1}^rC_i$ and from \Cref{eq:inter-dual}.
\end{proof}
We discuss faces and normal cones. Let $C\subset V$ be a convex subset. A subset 
$F\subset C$ is a {\em face} of $C$ if $F$ is convex and whenever 
$(1-\lambda)x+\lambda y\in F$ for some $\lambda\in(0,1)$ and $x,y\in C$, then $x$ 
and $y$ are also in $F$. A subset $F\subset C$ is an {\em exposed face} of $C$ if 
there is a linear functional $\ell\in V^\ast$ such that $x\in F$ if and only if 
$\ell(x)=\inf_{y\in C}\ell(y)$ for all $x\in C$. A point $x\in C$ is an 
{\em extreme point} (resp.~{\em exposed point}) of $C$ if $\{x\}$ is a face 
(resp.~exposed face) of $C$. Let $\R_+=\{\lambda\in\R:\lambda\geq 0\}$. 
If $x\in C$ is a nonzero point, and $\R_+x$ is a face (resp.~exposed face) of $C$, then 
the set $\R_+x$ is called an {\em extreme ray} (resp.~{\em exposed ray}) of $C$. 
\par
We denote the set of all faces of $C$ by $\cF(C)$. The family $\cF(C)$ is a complete 
lattice of finite length under the partial ordering of set inclusion \cite{Tam1985}. 
The {\em duality operator} of the closed convex cone $C$ is the map
\[
N_C:\cF(C)\to\cF(C^\vee),
\qquad
F\mapsto F^\perp\cap C^\vee.
\]
The map $N_C$ is antitone, that is to say, $F\subset G$ implies $N_C(F)\supset N_C(G)$ 
for all faces $F,G\in\cF(C)$. The image of $N_C$ is the set of nonempty exposed faces 
of $C^\vee$ and \cite[Prop.~2.4]{Tam1985} shows  $F\subset N_{C^\vee}\circ N_C(F)$ for 
all faces $F\in\cF(C)$. We have just confirmed that the pair of duality operators 
$N_C,N_{C^\vee}$ defines a \emph{Galois connection} between $\cF(C)$ and $\cF(C^\vee)$. 
Theorem~20 in Section~V.8 of \cite{Birkhoff1973} then proves the following assertion,
which we use in \Cref{thm:dual-hyp-cone}.
\par
\begin{Lem}\label{lem:lattice-iso}
Let $C\subset V$ be a closed convex cone. The duality operator $N_C$ restricts to 
an antitone lattice isomorphism from the set of nonempty exposed faces of $C$ to the 
set of nonempty exposed faces of $C^\vee$. The inverse isomorphism is the restricted
duality operator $N_{C^\vee}$.
\end{Lem}
It is easy to see that the exposed face $N_C(F)$ of $C^\vee$ is the (inner) 
{\em normal cone} to the closed convex cone $C$ at $F$,
\[
N_C(F)
=\{\ell\in V^\ast \colon \ell(y-x)\geq 0\; \forall y\in C \; \forall x\in F\},
\]
for all nonempty faces $F\subset C$. In addition, for every relative interior point $x$ of $F$, we have
\begin{equation}\label{eq:nc-x}
N_C(F)
=x^\perp\cap C^\vee
=\{\ell\in V^\ast \colon \ell(y-x)\geq 0\; \forall y\in C\},
\end{equation}
hence we also refer to $N_C(F)$ as the (inner) 
normal cone to $C$ at $x$.
\par
\subsection{Real Algebraic Geometry}
\label{sec:rag}
We are working in the setup of real algebraic geometry. A 
\emph{(real) affine variety} for us is a subset of $\C^n$ (for some $n\in\N$) 
that is defined by a finite number of polynomial equations 
$p_1 = p_2 = \cdots = p_r = 0$, $r\in\N$, with real coefficients 
$p_i\in\R[x_1,x_2,\dots,x_n]$, $i=1,\dots,r$. The affine varieties in $\C^n$ 
are the closed sets of the \emph{Zariski topology} on $\C^n$. So the Zariski 
closure of a set $S\subset \C^n$ is the smallest real affine variety containing 
$S$. A \emph{(real) projective variety} for us is a subset of projective space 
$\P^n$ that is defined by a finite number of homogeneous polynomial equations 
$p_1 = p_2 = \cdots = p_r = 0$, $r\in\N$, with real coefficients 
$p_i\in\R[x_0,x_1,\dots,x_n]$, $i=1,\dots,r$. The projective varieties in 
$\P^n$ are the closed sets of the \emph{Zariski topology} on $\P^n$.
Identifying points in $\P^n$ with lines in $\C^{n+1}$, a projective variety 
can be seen as an affine variety in $\C^{n+1}$ which is an algebraic cone. 
The \emph{affine cone} $\wh{S}$ over a subset $S\subset\P^n$  is the union of 
all lines in $\C^{n+1}$ spanned by a vector $(x_0,x_1,\dots,x_n)^T$ 
such that $(x_0:x_1:\dots:x_n)\in S$. Conversely, the projective variety 
$\P(X)\subset\P^n$ associated with an algebraic cone $X\subset\C^{n+1}$ 
consists of the points $(x_0:x_1:\dots:x_n)\in\P^n$ for which the vector 
$(x_0,x_1,\dots,x_n)^T$ is included in $X$. These notions are explained in 
introductory textbooks on algebraic geometry like \cite{Harris1992} with the 
caveat that affine and projective varieties are usually complex varieties, 
i.e.~defined by finitely many polynomial equations with complex coefficients. 
A point $x\in\P^n$ is {\em real} if the line $\wh{\{x\}}\subset\C^{n+1}$ 
contains a nonzero real point. We denote the set of real points of a subset 
$S\subset\C^n$ or $S\subset\P^n$ by $S(\R)$. 
The dual projective space $(\P^n)^\ast = \P((\C^{n+1})^\ast)$ is the projective 
space over the dual vector space so that the hyperplanes in $\P^n$ are in 
one-to-one correspondence with points in $(\P^n)^\ast$. We specify a functional  
$\ell\in(\P^n)^\ast$ in terms of its hyperplane 
\[
H\ =\ \cV(\ell)\ \subset\ \P^n
\]
by writing $\ell=[H]$. Identifying $((\P^n)^\ast)^\ast=\P^n$, a point $x\in\P^n$
defines a hyperplane in $(\P^n)^\ast$ which we denote by
\[
\cV(x)\ =\ \{y\in(\P^n)^\ast \colon x_0y_0 + x_1y_1 + \cdots + x_ny_n = 0 \}.
\]
\begin{Def}
Let $X\subset\P^n$ be an irreducible projective variety. We define the 
\emph{projective dual variety} $X^\ast$ of $X$ as the Zariski closure of the set 
of hyperplanes that are tangent to $X$ at some regular point, i.e.~the closure of
\[
\{[H] \in (\P^n)^\ast\colon T_x X\subset H \text{ for some } x\in X_\reg\}.
\]
\end{Def}
\par
An instructive and especially nice case of duality occurs for hypersurfaces 
defined by quadratic forms of full rank.
\begin{Exm}\label{exm:quadric}
Let $q \in\R[x_0,x_1,\dots,x_n]$ be a quadratic form and let $M_q$ be the real 
symmetric $(n+1)\times(n+1)$ matrix representing $q$, i.e.~with
\[
q = (x_0,x_1,\dots,x_n)M_q (x_0,x_1,\dots,x_n)^T.
\]
The projective variety $X = \cV(q)\subset \P^n$ is smooth if and only if the rank 
of $M_q$ is $n+1$. We compute the dual variety of $X$ under the assumption that 
$X$ is smooth. Let $x = (x_0:x_1:\dots:x_n)\in X$ be a point. The differential 
$\ell_x = 2x^TM_q\in(\P^n)^\ast$ of $q$ at $x$ defines the tangent hyperplane 
$T_xX = \{y\in\P^n\colon \ell_x(y) = 0\}$ to $X$ at $x$. In other words, the dual 
variety to $X$ is the Zariski closure of the set 
$\{\ell_x\colon x\in X\}\subset (\P^n)^\ast$. The condition $x\in X$ is 
$0 = x^TM_q x =  (\ell_x^T M_q^{-1}) M_q (M_q^{-1}\ell_x)$. We conclude that 
$X^\ast$ is the quadratic hypersurface defined by $M_q^{-1}$.
\end{Exm}
\par
For irreducible algebraic varieties, the famous biduality theorem holds.
\begin{Thm}[Biduality Theorem; see {\cite[Ch.~1, Thm.~1.1]{GKZ1994}}]\label{thm:biduality}
  If $X\subset \P^n$ is an irreducible projective 
  variety, then $(X^\ast)^\ast = X$ under the canonical identification of the 
  bidual of $\P^n$ with $\P^n$ itself.
\end{Thm}
This theorem has several useful consequences like the following. 
\par
\begin{Rem}\label{rem:dense-C}
Let $X\subset \P^n$ be an irreducible projective variety. For all points $x$ 
of $X$ in a dense subset in the Euclidean topology of $X$, the point $x$ is 
regular, the hyperplane $\cV(x)\subset(\P^n)^\ast$ is tangent to $X^\ast$ at a 
regular point $\ell$, and the hyperplane $\cV(\ell)\subset\P^n$ is tangent to $X$ 
at $x$. This is an application of the conormal variety $CN(X)$, defined as the 
Zariski closure of
\[
  CN_0(X)\  =\ \{(x,[H])\in\P^n\times (\P^n)^\ast\colon x\in X_\reg, T_xX\subset H\}.
\]
The projection $\pi_1 \colon CN_0(X) \to X_{\rm reg}$ is the conormal bundle of $X$, 
which shows that $CN_0(X)$ is an irreducible and smooth variety. The biduality 
theorem is often proven as a consequence of the fact that $CN(X) = CN(X^\ast)$, see 
\cite[Chapter 1]{GKZ1994}. The biduality theorem implies that the subset
\[
  U\ =\ \{(x,[H])\in CN_0(X) \colon [H]\in (X^\ast)_\reg\}
\]
is a non-empty Zariski open subset of $CN(X)$. If $x\in X_{\rm reg}$ and
$[H]\in (X^\ast)_{\rm reg}$ are regular points, then $H$ is tangent to $X$ at $x$  
if and only if $\cV(x)$ is tangent to $X^\ast$ at $[H]$, see 
\cite[Thm.~1.7(b)]{Tevelev2005}. This shows 
\[
  U\ =\ CN_0(X) \cap CN_0(X^\ast).
\]
By definition, the right-hand side consists of pairs $(x,[H])$ of regular points 
$x\in X_{\rm reg}$ and $[H]\in (X^\ast)_{\rm reg}$ such that $\cV(x)$ is tangent to 
$X^\ast$ at $[H]$ and $H$ is tangent to $X$ at $x$. Since $U$ is dense in $CN(X)$ 
in the Euclidean topology \cite[Thm.~2.33]{Mumford1976}, the claim follows as the 
projection from $CN(X)$ to the first factor $X$ is continuous and surjective.
\end{Rem}
When passing to real points, the direct analogue of Remark \ref{rem:dense-C} fails: 
The set of regular real points $X_\reg(\R)$ of an irreducible projective variety
$X\subset \P^n$ may not be dense in $X(\R)$ with respect to the Euclidean topology, 
even if it is non-empty, see for example \cite[Section 3.1]{BCR} or 
\Cref{ex:ChienNakazato} below. This is addressed in the following remark.
\par
\begin{Rem}\label{rem:bidualitytangency}
Let $X\subset \P^n$ be an irreducible projective variety. For all real regular points 
$x$ of $X$ in a dense subset of $X_\reg(\R)$ in the Euclidean topology, the hyperplane 
$\cV(x)\subset(\P^n)^\ast$ is tangent to $X^\ast$ at a real regular point $\ell$ and 
the hyperplane $\cV(\ell)\subset\P^n$ is tangent to $X$ at $x$. 
\par
This claim is trivial if $X$ has no regular real points. We resume the discussion from 
\Cref{rem:dense-C} assuming that $X$ does contain a regular real point. The variety 
$CN_0(X)$ is smooth and contains real points, since $X$ contains smooth real points. 
Since $CN_0(X)\setminus CN_0(X^\ast)=CN_0(X)\setminus U$ is a 
Zariski closed proper subset 
relative to $CN_0(X)$, it is of lower dimension.
As $CN_0(X)(\R)$ is a real analytic manifold of dimension $\dim(CN_0(X))$, see \cite[Prop.~3.3.11]{BCR},
the set $U(\R)$ is dense in $CN_0(X)(\R)$ in the Euclidean topology.
This proves the claim, because the projection of $CN_0(X)(\R)$ onto the first factor 
$X$ is onto $X_\reg(\R)$. 
\end{Rem}
\begin{Def}\label{def:central-point}
We call a real point $x$ of an algebraic variety $X\subset \C^n$ 
\emph{central} if it is in the Euclidean closure of the set of regular and real 
points of $X$, i.e.~if $x$ is in the Euclidean closure of $X_\reg(\R)$. 
\end{Def}
\begin{Rem}\label{rem:central-pt-local-dim}
For an irreducible algebraic variety $X\subset\C^n$, a point $x\in X(\R)$ is central 
if and only if the local dimension of $x$ in $X(\R)$ is equal to $\dim(X)$, see 
\cite[Prop.~7.6.2]{BCR}.
\end{Rem}
%
%
\section{Dual Hyperbolicity Cones}
\label{sec:duality-cones}
We discuss a result by the second author \cite{Sinn2015} more explicitly. The result 
is that the dual convex cone to a hyperbolicity cone is the convex cone generated by 
a particular semi-algebraic cone. This semi-algebraic cone is the Euclidean closure 
of the cone of regular real points on the dual variety to the hyperbolic hypersurface 
that lie in the right half-space. The algebraic boundary of the hyperbolicity cone 
allows us to simplify this semi-algebraic cone. We prove a stronger result for 
three-dimensional hyperbolicity cones.
\par
A homogeneous polynomial $p\in\R[x_0,x_1,\dots,x_n]$ of degree $d$ is called 
\emph{hyperbolic} with respect to a fixed point $e\in\R^{n+1}$ if 
$p(e)\neq 0$ and the polynomial $p(te-a)$ in one variable $t$ has only real 
roots for every point $a\in\R^{n+1}$. Without loss of generality, we fix the 
sign at $e$ and always assume $p(e)>0$. The roots of $p(te-a)$ are sometimes 
called the eigenvalues of $a$ with respect to $p$ and $e$, in analogy with 
characteristic polynomials of Hermitian matrices. Given any such polynomial 
$p$, the set
\[
C_e(p)\ =\ \bigl\{ a\in\R^{n+1} 
\colon \mbox{all roots of $p(te-a)$ are non-negative} \bigr\}
\]
is a closed convex cone called the \emph{(closed) hyperbolicity cone of $p$ with 
respect to $e$}, 
and $e$ is an interior point of $C_e(p)$,
see \cite{Renegar2006}. Our goal is to describe the dual convex cone
\[
C_e(p)^\vee\ =\ \bigl\{ \ell\in (\R^{n+1})^\ast 
\colon \mbox{$\ell(x)\geq 0$ for all $x\in C_e(p)$} \bigr\}.
\]
An essential technique is projective duality. A general approach is described in 
the paper \cite{Sinn2015} by the second author. The goal of this section is to 
explain this method more explicitly for the special case of hyperbolicity cones.
\par
An important example of a hyperbolic polynomial is the determinant of a matrix 
pencil, i.e. $p=\det(x_0A_0+\cdots+x_nA_n)$ for Hermitian $d\times d$ matrices 
$A_0,\dots,A_n\in H_d$, which is hyperbolic with respect to $e=(e_0,\dots,e_n)^T$ 
provided the matrix $e_0A_0+\cdots+e_nA_n$ is positive definite. In this case, 
$C_e(p)$ is the \emph{spectrahedral cone} defined by $A_0,\dots,A_n$, 
see \Cref{sec:PSD-cone}. However, the discussion in this section does not require 
such a determinantal representation and we consider general hyperbolic polynomials.
\par
The proof of our main result, \Cref{thm:dual-hyp-cone}, makes use of the following 
\Cref{lem:multiplicity} on hyperbolic polynomials. For the sake of completeness, we 
include a short proof based on the Helton-Vinnikov theorem on determinantal 
representations of hyperbolic curves; see \cite[Lemma~2.4]{PlaumannVinzant2013} for 
a direct proof of a special case.
\begin{Def}
Let $f\in\R[x_0,x_1,\dots,x_n]$ be a polynomial and $x\in\C^{n+1}$ be a point. 
The \emph{multiplicity} of $x$ on $\cV(f)\subset\C^{n+1}$ is the smallest degree 
of a non-zero homogeneous term in the Taylor expansion of $f$ around $x$.
\end{Def}
\begin{Lem}\label{lem:multiplicity}
Let $p\in\R[x_0,\dots,x_n]$ be hyperbolic with respect to $e$ and let $x\in\R^{n+1}$.
If $x$ has multiplicity $m$ on $\cV(p)$, the hyperbolic hypersurface defined by $p$, 
then $t=0$ is a root of multiplicity $m$ of $p(x+te)\in\R[t]$. Moreover, $t=0$ is 
also a root of multiplicity $m$ of $p(x+t(e-x))\in\R[t]$.
\end{Lem}
\begin{proof}
If the multiplicity of $x$ on $\cV(p)$ is $m$, then 
$\frac{\partial^m}{\partial s^m}p(x+sy)|_{s=0}\neq 0$ for generic $y\in\R^{n+1}$. 
Fix such $y$ in the interior of $C_e(p)$ and consider the hyperbolic polynomial 
$p(rx+sy+te)$ in three variables $r,s,t$. By the Helton-Vinnikov theorem 
(\cite[Thm.~2.2]{HeltonVinnikov2007}), this polynomial has a determinantal 
representation
\[
  p(rx+sy+te) = \det(rA + sB + tC)
\]
with real symmetric matrices $A,B,C$, where $B$ and $C$ are positive definite, hence 
factor as $B=UU^T$ and $C=VV^T$, with $U$ and $V$ invertible. Now $s=0$ is a root of 
$p(x+sy)=\det(A+sB)$ of multiplicity $m$, which means that $U^{-1}A(U^T)^{-1}$ has 
$m$-dimensional kernel. But then so does $V^{-1}A(V^T)^{-1}$, hence the root $t=0$ of 
$p(x+te)=\det(A+tC)$ has multiplicity $m$ as well.
\par
The second part of the claim follows from the part we have just proved. Indeed, write 
$f(s,t) = p(sx + te)\in \R[s,t]$ which is the homogenization of $p(x+te)$ because 
$p(e)\neq 0$. Since $t = 0$ is a root of multiplicity $m$ of $p(x+te)$, we can write 
$p(x + te) = t^mq(t)$ with $q(0)\neq 0$. Since $p$ is hyperbolic with respect to $e$, 
the polynomial $q$ factors into linear terms over $\R$, say 
$q = c (1- \lambda_1 t)\cdot\ldots\cdot (1-\lambda_{d-m}t)$ for some nonzero $c\in\R$. 
Here we used $q(0)\neq 0$ because the $\lambda_i$ are the reciprocals of the roots of 
$q$. So we get that $f(s,t) = c \cdot t^m \prod (s-\lambda_i t)$. The polynomial 
$p(sx+t(e-x))$ is the same as $f(s-t,t) = c \cdot t^m \prod (s- (\lambda_i + 1)t)$. 
Dehomogenizing this again shows that $t=0$ is a root of multiplicity $m$ of 
$p(x+t(e-x))$.
\end{proof}
\begin{Cor}\label{Cor:multiplicity}
If $x$ is a regular real point of a hyperbolic hypersurface $\cV(p)$, then the line incident with $x$ and the hyperbolicity direction $e$ is not tangent to $\cV(p)$ at 
$x$, i.e.~is not contained in $T_x\bigl(\cV(p)\bigr)$.
\end{Cor}
\begin{proof}
If a line $L$ is tangent to $\cV(p)$ at $x$, then the multiplicity of $x$ in 
$L\cap \cV(p)$ is greater than the multiplicity of $x$ in $\cV(p)$. This is impossible 
if $e\in L$, by the previous \Cref{lem:multiplicity}.
\end{proof}
%
%
%
We prove basics from differential and convex geometry regarding hyperbolicity cones.
\par
\begin{Lem}\label{lem:some-hyper}
Let $p\in\R[x_0,\dots,x_n]$ be an irreducible hyperbolic polynomial with respect to 
$e\in\R^{n+1}$. Then $M=\partial C_e(p)\cap\cV(p)_\reg$ is an $n$-dimensional real 
analytic manifold, which is open and dense in the Euclidean boundary $\partial C_e(p)$ 
of $C_e(p)$ in the Euclidean topology. The (inner) normal cone to $C_e(p)$ at any 
point $x\in M$ is the ray $\R_+\ell$, where 
$\ell=\nabla p(x)^T\in C_e(p)^\vee\subset(\R^{n+1})^\ast$. 
\end{Lem}
\begin{proof}
Since $\partial C_e(p)\subset \cV(p)(\R)$ and since the set of singular points of 
$\cV(p)(\R)$ is a variety of dimension at most $n-1$, see \cite[Prop.~3.3.14]{BCR},
the complement $M=\partial C_e(p)\cap\cV(p)_\reg$ is open and dense in $\partial C_e(p)$ 
in the Euclidean topology. As $\cV(p)_\reg(\R)$ is an analytic 
manifold of dimension $n$, see \cite[Prop.~3.3.11]{BCR}, and as the eigenvalues 
depend continuously on $x\in\R^{n+1}$, see \cite{Renegar2006}, the set $M$ is 
an analytic manifold of dimension $n$.
\par
Let $x\in M$. As $x$ is a regular point of $\cV(p)$, the functional 
$\ell=\nabla p(x)^T$ is non-zero. Hence, \Cref{lem:multiplicity}, case $m=1$, shows 
$\ell(e)\neq 0$. Since $M$ is an $n$-dimensional analytic manifold included in 
$C_e(p)$, the normal cone $N(x)$ of $C_e(p)$ at $x$ is a subset of the line $\R\ell$, 
hence $N(x)=\R_+\ell$ or $N(x)=-\R_+\ell$ as $x\in\partial C_e(p)$. The derivative 
polynomial $p'\in\R[x_0,\dots,x_n]$, defined by
\[
p'(y)=\tfrac{\partial}{\partial t}p(y+te)|_{t=0}=\nabla p(y)^Te,
\qquad y\in\R^{n+1},
\]
is hyperbolic with respect to $e$ and $C_e(p)\subset C_e(p')$ holds, see 
\cite{Renegar2006}. This proves $\ell(e)=p'(x)>0$ and rules out $N(x)=-\R_+\ell$, 
as $N(x)\subset C_e(p)^\vee$ by \Cref{eq:nc-x}. This proves the claim.
\end{proof}
The second author obtained \Cref{thm:dual-hyp-cone} and 
\Cref{cor:dual-hyp-cone-bound} below in \cite[Example~3.15]{Sinn2015}. 
We slightly abuse notation in the following statements. If $X\subset\C^{n+1}$ is an 
algebraic cone, then we write $X^\ast\subset (\C^{n+1})^\ast$ for the affine cone over 
the projective dual variety $\P(X)^\ast$ of $\P(X)$. Let $H_{e,+}\subset(\R^{n+1})^\ast$ 
denote the half-space $H_{e,+} = \{\ell\in (\R^{n+1})^\ast\colon \ell(e)\geq 0\}$ for 
nonzero $e\in\R^{n+1}$. 
\par
\begin{Thm}[Sinn \cite{Sinn2015}]\label{thm:dual-hyp-cone}
Let $p\in\R[x_0,x_1,\dots,x_n]$ be an irreducible hyperbolic polynomial with respect 
to $e\in\R^{n+1}$. Then we have
\[
C_e(p)^\vee\ 
=\  \cl \Big( \co \big( (\cV(p)^\ast)_\reg(\R) \cap H_{e,+} \big) \Big).
\]
\end{Thm}
\begin{proof}
We prove the inclusion ``$\supset$''. Since $C_e(p)^\vee$ is a closed convex cone, it 
is enough to show that $S=(\cV(p)^\ast)_\reg(\R) \cap H_{e,+}$ is contained in $C_e(p)^\vee$. 
By \Cref{rem:bidualitytangency}, it is enough to prove $\ell\in C_e(p)^\vee$ for all 
$\ell\in S$ such that the hyperplane $\cV(\ell)$ is tangent to $\cV(p)$ at a regular 
real point $x$, i.e.~$\ell = \nabla p(x)^T\in (\R^{n+1})^\ast$. By 
\Cref{lem:multiplicity} above, case $m=1$, we have 
$0\neq\frac{\partial}{\partial t}p(x+te)|_{t=0} = \ell(e)$ for such an $\ell$ and 
hence $e\not\in\cV(\ell)$. As $p$ is hyperbolic with respect to every interior point 
of the hyperbolicity cone \cite{Renegar2006}, it follows that the interior of $C_e(p)$ 
is disjoint from $\cV(\ell)$. This means that $\ell$ has constant sign on $C_e(p)$. 
Since $\ell\in H_{e,+}$, it follows that $\ell \in C_e(p)^\vee$. 
\par
The inclusion ``$\subset$'' follows from \cite[Coro.~3.14]{Sinn2015}. We repeat the 
argument, adapted to our setup, for completeness. Since $C_e(p)^\vee$ is the convex 
cone generated by its extreme rays and the right-hand side is a convex cone, it 
suffices to prove that every extreme ray of $C_e(p)^\vee$ is contained in the 
right-hand side. By Straszewicz's Theorem \cite[Thm.~18.6]{Rockafellar1970}, which 
says that every extreme ray is a limit of exposed rays, it suffices to prove the claim 
for every exposed ray of $C_e(p)^\vee$, because the right-hand side is closed. 
\par
Let $\R_+ \ell$ be an exposed ray of the convex cone $C_e(p)^\vee$. It is enough to 
show that $\ell$ is a central point of $\cV(p)^\ast(\R)$, i.e.~$\ell$ lies in the 
(Euclidean) closure of $(\cV(p)^\ast)_{\rm reg}(\R)$. Let 
$F_\ell = \{x\in C_e(p)\colon \ell(x) = 0\}$ be the exposed face of $C_e(p)$ 
corresponding to $\ell$ and let $x$ be a point in the relative interior of $F_\ell$. 
As proven in \Cref{lem:some-hyper}, the analytic manifold 
$M=\partial C_e(p)\cap\cV(p)_\reg$ is dense in the Euclidean boundary $\partial C_e(p)$ 
of the hyperbolicity cone $C_e(p)$. Hence, there is a sequence $(y_j)\subset M$ 
converging to $x$. \Cref{rem:bidualitytangency} shows that after slightly moving the 
members of the sequence $(y_j)$ within $M$ without changing the limit $x$, the 
hyperplane $\cV(y_j)$ is tangent to $\cV(p)^\ast$ at a regular real point 
$\ell_j\in(\cV(p)^\ast)_\reg(\R)$ and $\cV(\ell_j)$ is tangent to $\cV(p)$ at $y_j$ 
for all $j\in\N$. 
\par
After scaling $\ell_j$ with a nonzero real number, we have $\ell_j=\nabla p(y_j)^T$. 
Hence, the ray $\R_+\ell_j$ lies in the dual convex cone $C_e(p)^\vee$ by 
\Cref{lem:some-hyper} for all $j\in\N$. After normalizing and passing to a subsequence, 
the sequence $(\ell_j)$ converges to a point $\ell'$ in the compact unit sphere of 
$(\R^{n+1})^\ast$. We have $\ell'\in C_e(p)^\vee$ and $\ell'(x) = 0$, the latter as 
\begin{align*}
\ell'(x)
= \ell_j (x-y_j) + (\ell'-\ell_j) y_j + (\ell'-\ell_j) (x-y_j)
\end{align*}
holds for all $j\in\N$ and since $\ell_j\to\ell'$ and $y_j\to x$ as $j\to\infty$.
Since $\R_+ \ell$ is an exposed ray of $C_e(p)^\vee$, the lattice isomorphism of
\Cref{lem:lattice-iso} and \Cref{eq:nc-x} show
$\R_+ \ell=F_\ell^\perp\cap C_e(p)^\vee=x^\perp\cap C_e(p)^\vee$.
This proves $\R_+ \ell=\R_+ \ell'$. Hence $\ell$ is a central point of 
$\cV(p)^\ast(\R)$.
\end{proof}
In other words, \Cref{thm:dual-hyp-cone} says that the dual convex cone to the 
hyperbolicity cone $C_e(p)$ is the closed convex cone generated by the regular 
real points of the dual variety $\cV(p)^\ast$ lying in the appropriate half-space 
$H_{e,+}$.
\par
\begin{figure}[ht!]%
a)\includegraphics[height=6cm]{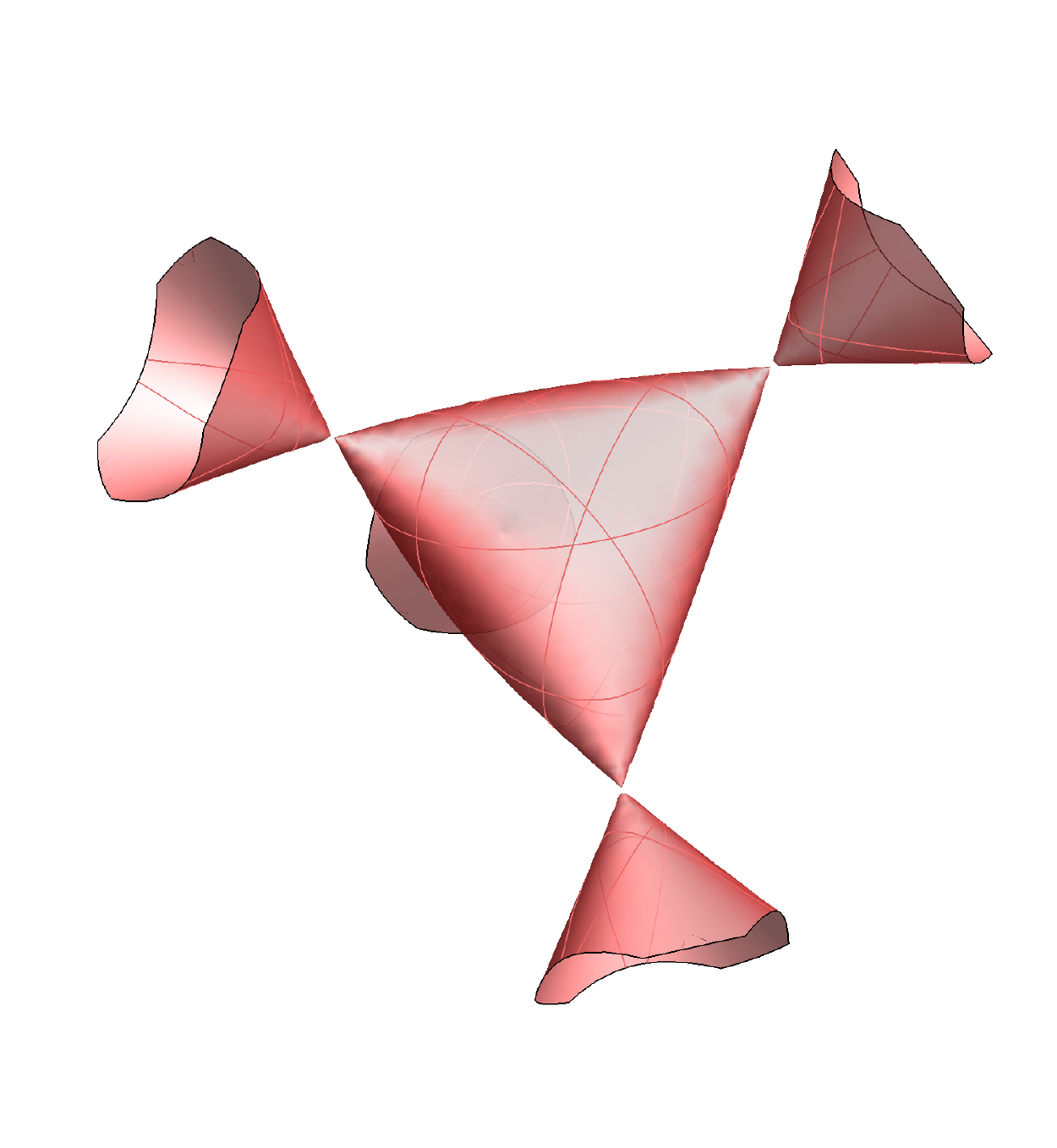}%
b)\includegraphics[height=6cm]{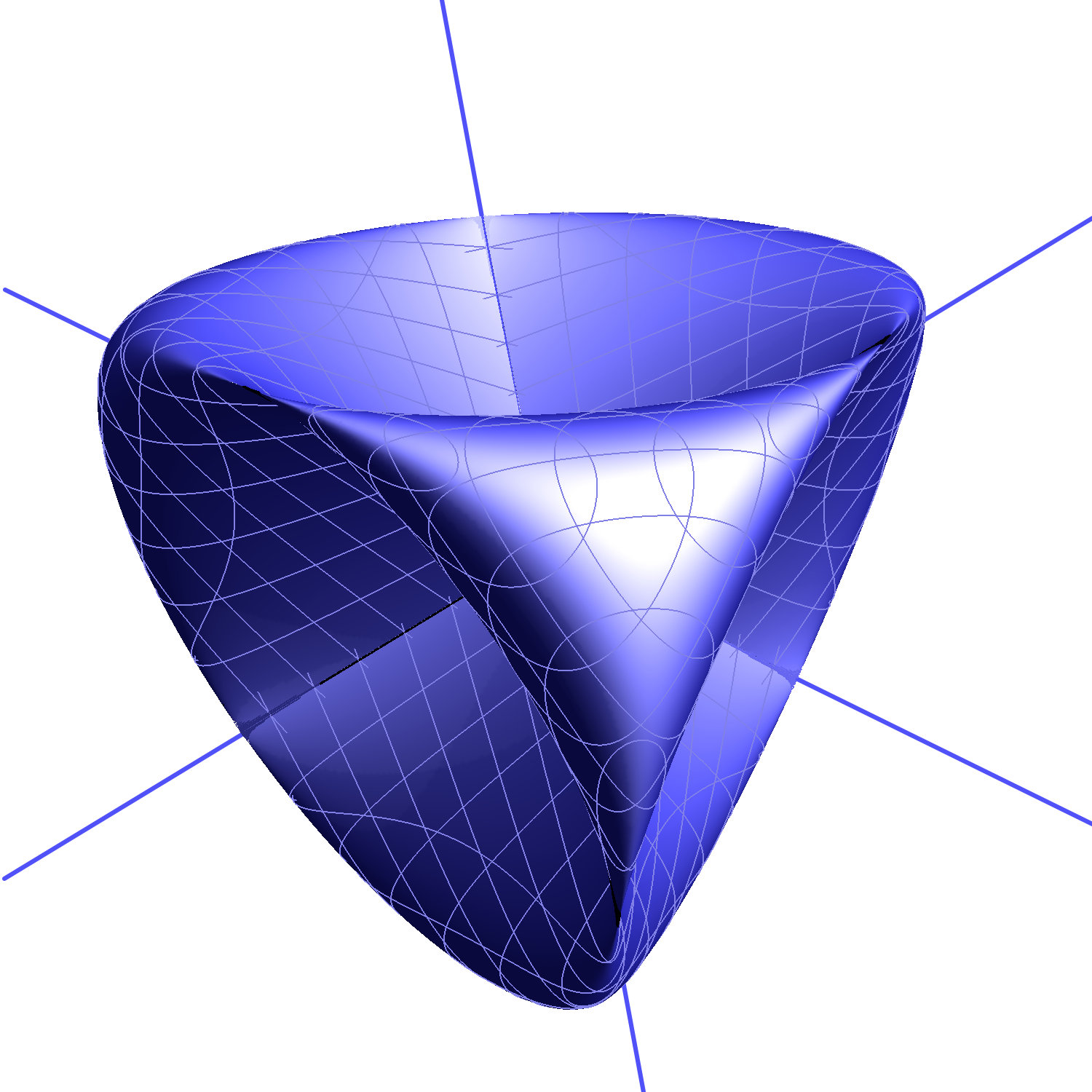}%
\caption{%
a) Cayley cubic.
b) Steiner surface with three singular lines.}
\label{fig:steiner}
\end{figure}
The well-known Steiner surface explains why singular points of the 
dual variety $\cV(p)^\ast$ have to be excluded from the statement of 
\Cref{thm:dual-hyp-cone}.
\par
\begin{Exm}\label{ex:Cayley}
  The Cayley cubic is the cubic hypersurface in $\P^3$ defined by the polynomial
  \[
 p=\det
    \begin{pmatrix}
      x_0 & x_1 & x_3 \\
      x_1 & x_0 & x_2 \\
      x_3 & x_2 & x_0
    \end{pmatrix}.
  \]
This polynomial is irreducible and hyperbolic with respect to the point $(1,0,0,0)$ 
in $\R^4 = \{(x_0,x_1,x_2,x_3)\}$. Its hyperbolicity cone is the homogenization of the 
elliptope $\mathcal{E}_3$, which is the feasible set of the Goemans-Williamson 
semidefinite relaxation of the MAX-CUT problem (for graphs with three vertices). 
\par
The dual convex cone is the closed convex cone generated by the regular real points 
with $y_0 > 0 $ on the Steiner surface given by the equation
  \[
    q=y_1^2y_2^2 + y_1^2y_3^2 + y_2^2y_3^2 - 2 y_0y_1y_2y_3.
  \]
The singular locus of this quartic surface is the union of three real lines in 
$(\P^3)^\ast$, which are not contained in the dual convex cone. 
See \Cref{fig:steiner}, where we draw the real affine parts of the varieties 
$\cV(p)$ and $\cV(q)$, that is to say, the set of points $(x_1,x_2,x_3)^T\in\R^3$ 
for which $(1:x_1:x_2:x_3)$ lies in $\cV(p)$ in drawing a) and the set of points 
$(y_1,y_2,y_3)\in(\R^3)^\ast$ for which $(1:y_1:y_2:y_3)$ lies in $\cV(q)$ in 
drawing b).
\end{Exm}
Convex geometry suffices to generalize \Cref{thm:dual-hyp-cone} from irreducible 
polynomials to arbitrary hyperbolic polynomials.
\par
\begin{Cor}\label{cor:dual-hyp-cone}
Let $p\in\R[x_0,x_1,\dots,x_n]$ be a hyperbolic polynomial with respect 
to $e\in\R^{n+1}$. Let $X_1,X_2,\dots,X_r$ be the irreducible components 
of the hyperbolic hypersurface $\cV(p)$. Then we have
\begin{equation}\label{eq:dual-hyp-cone}
C_e(p)^\vee\ 
= \co\big(\cl(S_1\cup S_2\cup\dots\cup S_r)\big),
\end{equation}
where $S_i = (X_i^\ast)_\reg(\R) \cap H_{e,+}$ for $i=1,\dots,r$. 
\end{Cor}
\begin{proof}
If $p=p_1^{m_1}p_2^{m_2}\dots p_r^{m_r}$ is a factorization of $p$ into irreducible 
factors, then $C_e(p)=C_e(p_1)\cap C_e(p_2)\cap\dots\cap C_e(p_r)$. The claim follows
from \Cref{thm:dual-hyp-cone} and from \Cref{eq:inter-dual} in 
\Cref{pro:dual-of-intersection}.
\end{proof}
\begin{Rem}
\begin{enumerate}
\item
The cones $S_i$ in \Cref{cor:dual-hyp-cone}, their union $S_1\cup S_2\cup\dots\cup S_r$,
and the Euclidean closure $\cl (S_1 \cup S_2 \cup\dots\cup S_r)$, see 
\cite[Prop.~2.2.2]{BCR}, are semi-algebraic sets. Hence, the closed convex 
cone $C_e(p)^\vee$ is the convex cone generated by the semi-algebraic cone
$\;\cl (S_1 \cup S_2 \cup\dots\cup S_r)$.
\item
The Euclidean closure of the set $S_i$ in \Cref{cor:dual-hyp-cone} is the set of 
central points of $X_i^\ast(\R)$ lying in $H_{e,+}$ for all $i = 1,2,\dots,r$. 
Hence, writing $\mathrm{cent}(X)$ for the set of central real points of $X$, we 
can rephrase \Cref{eq:dual-hyp-cone} as
\[
C_e(p)^\vee\ 
=\ \co \left( \bigcup_{i=1}^r \mathrm{cent}(X_i^\ast(\R))\cap H_{e,+} \right).
\]
\end{enumerate}
\end{Rem}
We can strengthen \Cref{thm:dual-hyp-cone} to the homogeneous version of 
Kippenhahn's Theorem, the statement of \Cref{thm:dim3}, if the hyperbolicity cone
has dimension three. Chien and Nakazato \cite{ChienNakazato2010} observed that 
this stronger version is false in higher dimensions, see \Cref{ex:ChienNakazato}.
\par
\begin{Thm}\label{thm:dim3}
Let $p\in \R[x_0,x_1,x_2]$ be an irreducible hyperbolic polynomial with respect to 
$e\in\R^3$. Then $C_e(p)^\vee = \co\big(\cV(p)^\ast(\R)\cap H_{e,+}\big)$.
\end{Thm}
\begin{proof}
The inclusion $C_e(p)^\vee \subset \co\big(\cV(p)^\ast(\R)\cap H_{e,+}\big)$ follows 
from \Cref{thm:dual-hyp-cone} and \Cref{lem:cn-ccc}. We prove the opposite inclusion 
by contradiction based on two observations. Let $\ell$ be a nonzero functional that
lies in $\co\big(\cV(p)^\ast(\R)\cap H_{e,+}\big)$ but not in $C_e(p)^\vee$. 
\par
First, the hyperplane $\cV(\ell)\subset\C^3$ intersects the interior of the 
hyperbolicity cone $C_e(p)$, as it holds for all functionals 
$\hat{\ell}\in H_{e,+}\setminus C_e(p)^\vee$. Since $\hat{\ell}\in H_{e,+}$, we have
$\hat{\ell}(e)\geq 0$, and since $\hat{\ell}\not\in C_e(p)^\vee$ there is a point 
$x\in C_e(p)$ such that $\hat{\ell}(x)<0$. Hence, there is $\lambda\in[0,1)$ such 
that the point $y=(1-\lambda)e+\lambda x$ lies on $\cV(\hat{\ell})$. As $e$ is an 
interior point of $C_e(p)$ so is $y$ \cite[Thm.~6.1]{Rockafellar1970}.
\par
Secondly, the line $\cV(\ell)\subset\P^2$ is tangent to $\cV(p)$ at a point 
$x\in\cV(p)$. Applying \Cref{rem:dense-C} to the projective variety $X=\cV(p)^\ast$, 
we can choose a sequence of regular points $(\ell_j)$ in $\cV(p)^\ast$ converging to
$\ell$ and a sequence of regular points $(x_j)$ of $\cV(p)$ such that the line 
$\cV(\ell_j)$ is tangent to $\cV(p)$ at $x_j$ for all $j\in\N$. Because of the 
compactness of the projective space $\P^2$, we can assume $(x_j)$ converges to a 
point $x\in\cV(p)$. Since $\cV(\ell_j)$ is tangent to $\cV(p)$ at $x_j$, the line 
$\cV(\ell)$ is tangent to $\cV(p)$ at $x$, see \cite[Sec.~8.2]{Fischer2001}.
\par
The real line $\cV(\ell)\subset\P^2$ intersects the hyperbolicity cone $C_e(p)$ 
in an interior point by the first observation. Since the polynomial $p$ is 
hyperbolic with respect to this interior point and since the point $x\in\cV(p)$ 
constructed above lies on the line $\cV(\ell)$, it follows that $x$ is a real point.
\Cref{lem:multiplicity} then shows that $\cV(\ell)$ is not tangent to $\cV(p)$ at 
$x$, which contradicts the second observation.
\end{proof}
Again, convex geometry suffices to generalize \Cref{thm:dim3} from irreducible 
polynomials to arbitrary hyperbolic polynomials.
\par
\begin{Cor}\label{cor:dim3}
Let $p\in \R[x_0,x_1,x_2]$ be a hyperbolic polynomial with respect to $e\in\R^3$. 
Then $C_e(p)^\vee = \co  (\cV(p)^\ast(\R)\cap H_{e,+})$.
\end{Cor}
\begin{proof}
If $p=p_1^{m_1}p_2^{m_2}\dots p_r^{m_r}$ is a factorization of $p$ into irreducible 
factors, then 
\[
C_e(p)=C_e(p_1)\cap C_e(p_2)\cap\dots\cap C_e(p_r)
\]
and 
\[
\cV(p)^\ast=\cV(p_1)^\ast\cup\cV(p_2)^\ast \cup\dots\cup\cV(p_r)^\ast.
\]
The claim follows from \Cref{thm:dim3} and \Cref{eq:inter-dual} in 
\Cref{pro:dual-of-intersection}.
\end{proof}
Not all components in the union 
$\cV(p)^\ast=\cV(p_1)^\ast\cup\cV(p_2)^\ast \cup\dots\cup\cV(p_r)^\ast$ are needed 
in the statements of \Cref{cor:dim3} and \Cref{cor:dual-hyp-cone}. The selection can 
be described as follows.
\par
\begin{Def}\label{def:alg-boundary}
Let $S\subset \R^n$ be a semi-algebraic set. The \emph{algebraic boundary} 
of $S$, denoted $\partial_a S$, is the Zariski closure in $\C^n$  of the 
Euclidean boundary $\partial S$ of $S$. 
\end{Def}
Determining the algebraic boundary of the hyperbolicity cone of a hyperbolic
polynomial $p\in\R[x_0,x_1,\dots,x_n]$ amounts to computing a factorization of 
$p$ into irreducible factors and picking the correct subset of the factors.
\par
\begin{Rem}\label{rem:boundary-hyperbolicity}
\begin{enumerate}
\item
Let $p\in\R[x_0,x_1,\dots,x_n]$ be irreducible and hyperbolic with respect to 
$e$. The algebraic boundary of the hyperbolicity cone $C_e(p)$ is the algebraic 
hypersurface $\cV(p) = \{x\in\C^{n+1}\colon p(x) = 0\}$.
\item
If $p$ is hyperbolic with respect to $e$, but factors as $p = p_1 p_2\dots p_r$ 
into irreducible factors, then 
$\cV(p) = \cV(p_1)\cup \cV(p_2)\cup \dots\cup \cV(p_r)$ is the decomposition of 
the hypersurface $\cV(p)$ into its irreducible components. The algebraic boundary 
of $C_e(p)$ is a union of some, not necessarily all, of the irreducible 
hypersurfaces $\cV(p_i)$. The hypersurfaces in this union are the irreducible 
components of $\partial_a C_e(p)$.
\item
If $p$ is squarefree, i.e.~$p_j\neq p_i$ for $i\neq j$, then the factors $p_i$ that 
contribute to the algebraic boundary of $C_e(p)$ are exactly those with the property 
that the hyperbolicity cone of $\prod_{j\neq i}p_j$ is strictly larger than $C_e(p)$, 
because $C_e(p) = \bigcap_i C_e(p_i)$.
\end{enumerate}
\end{Rem}
\begin{figure}[ht!]%
a)\includegraphics[height=5cm]{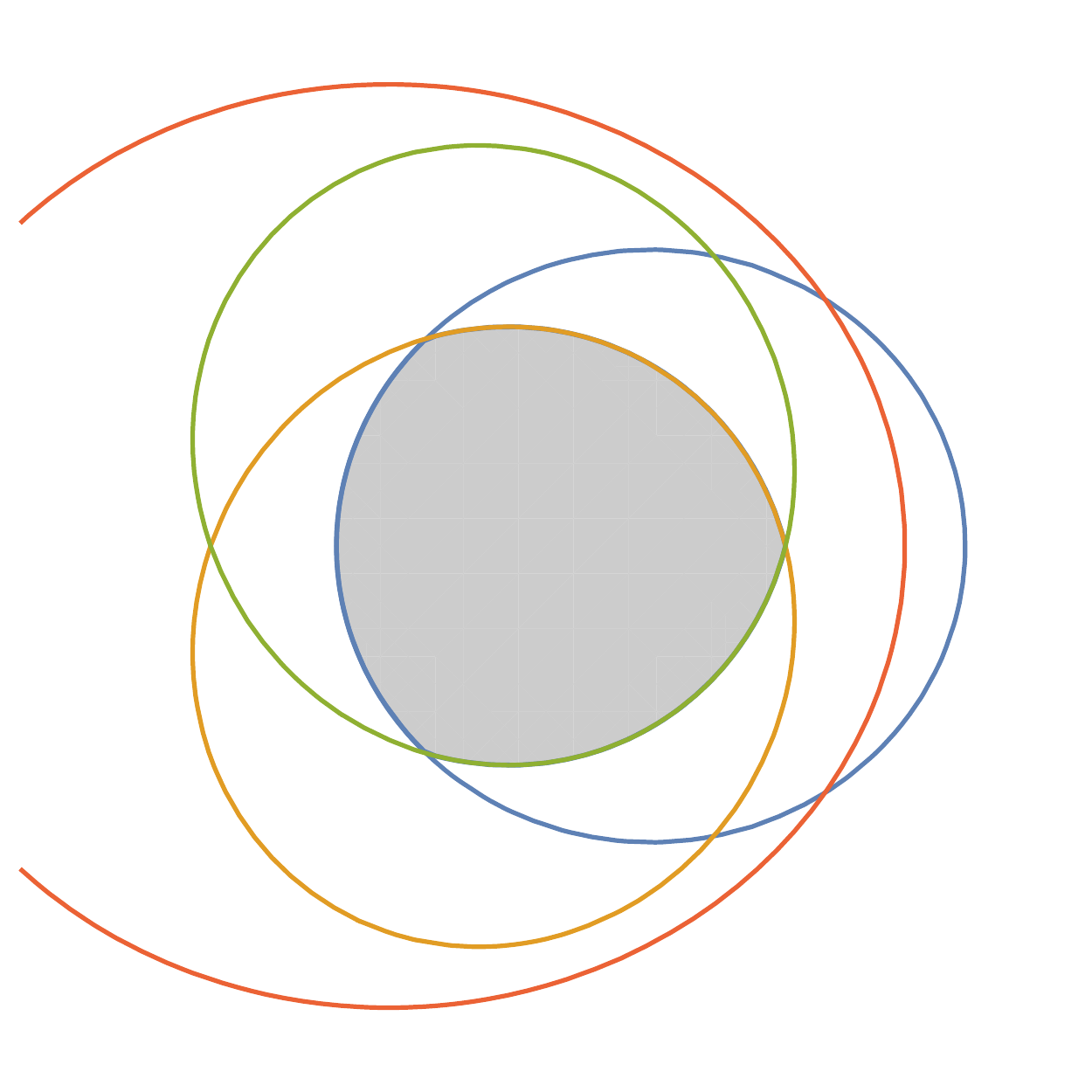}%
b)\includegraphics[height=5cm]{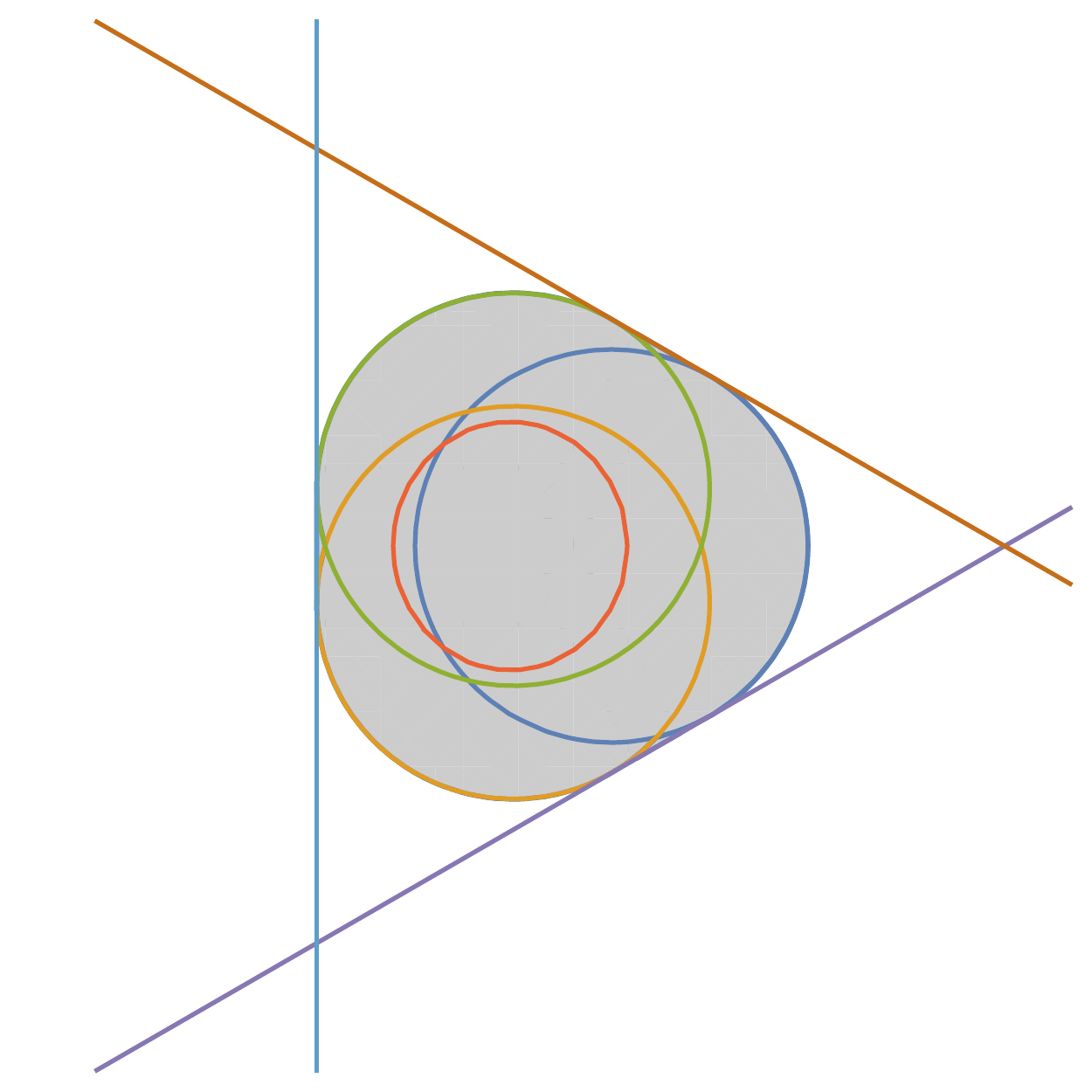}%
\caption{%
a) Intersection of four filled ellipses (gray area).
b) Convex hull of the dual ellipses (gray area). 
}
\label{fig:four-ellipses}
\end{figure}
%
\par
\begin{Exm}\label{ex:four_ellipses}
Consider the four ellipses depicted in \Cref{fig:four-ellipses} a). The intersection 
of the filled ellipses is an area with nonempty interior, which is isometric to the 
base of a hyperbolicity cone \emph{via} the embedding $\R^2\to\R^3$, 
$(x_1,x_2)^T\mapsto(1,x_1,x_2)^T$. Note that the red ellipse does not contribute to 
the algebraic boundary of the hyperbolicity cone, but the other ellipses do. 
Homogenizing the polynomials, we obtain four conics in $\P^2$ from the ellipses. The 
real affine parts of the dual conics, \emph{via} the embedding 
$(\R^2)^\ast\to(\P^2)^\ast$, $(y_1,y_2)\mapsto(1:y_1:y_2)$, are depicted in 
\Cref{fig:four-ellipses} b). Their convex hull is a base of the dual convex cone to 
the hyperbolicity cone. We return to these bases in \Cref{sec:duality-selfdual}. Note 
that the red ellipse in \Cref{fig:four-ellipses} b) is redundant, as the other three 
ellipses generate the same convex hull as all four together.
\end{Exm}
\begin{Cor}[Sinn~\cite{Sinn2015}]\label{cor:dual-hyp-cone-bound}
Let $p\in\R[x_0,x_1,\dots,x_n]$ be a hyperbolic polynomial with respect to 
$e\in\R^{n+1}$. Let $X_1,X_2,\dots,X_r$ be the irreducible components of the algebraic boundary $\partial_a C_e(p)$. Then we have
\[
C_e(p)^\vee\ 
= \co\big(\cl(S_1\cup S_2\cup\dots\cup S_r)\big),
\]
where $S_i = (X_i^\ast)_\reg(\R) \cap H_{e,+}$ for $i=1,\dots,r$. 
\end{Cor}
\begin{proof}
This follows from Part 3) of \Cref{rem:boundary-hyperbolicity} and 
\Cref{cor:dual-hyp-cone}.
\end{proof}
\begin{Rem}\label{rem:zariski-extreme}
The algebraic boundary of the hyperbolicity cone $C_e(p)$ is algebraically the optimal description of the set of extreme rays of $C_e(p)^\vee$.
Using the same techniques and ideas that are presented in this paper, one can argue that the Zariski closure of the set of extreme rays of $C_e(p)^\vee$ is the union of the dual varieties $X_i^\ast$ 
for which the varieties $X_i$ belong to the algebraic boundary 
$\partial_a C_e(p)$, see \cite[Corollary~3.5]{Sinn2015}. 
\end{Rem}
The statements above in this section require employing the duality theory of real 
algebraic geometry, which is more subtle than the duality theory of algebraic geometry. 
The reason is that the set of Hermitian matrices is just a real vector space, not a 
complex vector space.
\par
\begin{Rem}\label{rem:complex-duality}
The image of the set $H_d^{(1)}$ of Hermitian matrices of rank $1$ under the projection 
map 
\[
  \pi(M) = (\scp{M,A_1},\scp{M,A_2},\ldots,\scp{M,A_n})
\]
from $H_d$ to $\R^n$ is of interest to us (see \Cref{sec:qm}), as its convex 
hull is the joint numerical range. The complexification of the real vector space $H_d$ 
is the complex vector space of all complex $d\times d$ matrices. The Zariski closure
of $H_d^{(1)}$ is the variety $R_d^{(1)}$ of complex $d\times d$ matrices of rank at 
most $1$. Therefore, the images of $H_d^{(1)}$ and $R_d^{(1)}$ under the map $\pi$ have 
the same Zariski closure. To evaluate $\scp{M,A_i}$ at a complex matrix $M$, we write 
$M = {\rm Re}(M) + \ii {\rm Im}(M)$, where ${\rm Re}(M)$ and ${\rm Im}(M)$ are 
Hermitian, and define $\scp{M,A_i} = \tr({\rm Re}(M) A_i) + \ii \tr({\rm Im}(M)A_i)$ 
(which is $\C$-linear). If the map $\pi$ restricted to $R_d^{(1)}$ is an isomorphism, 
then Prop.~4.1 in \cite{GKZ1994} implies roughly speaking that $\pi(R_d^{(1)})$ 
is projectively dual to the intersection of the orthogonal complement of the kernel of 
$\pi$ and the dual variety of $R_d^{(1)}$. This dual variety is the determinantal 
hypersurface in the space of complex $d\times d$ matrices. The assumption that $\pi$ 
restricted to $R_d^{(1)}$ is an isomorphism is generically satisfied provided that $n$ 
is sufficiently large relative to the size $d$ of the matrices; specifically the kernel 
of $\pi$ must not intersect the secant variety of $R_d^{(1)}$, which is the variety of 
matrices of rank at most $2$.
\end{Rem}
%
%
\section{Bases of Dual Hyperbolicity Cones}
\label{sec:duality-selfdual}
The results of \Cref{sec:duality-cones} carry over from the dual convex cone of a 
hyperbolicity cone to all its bases, by replacing the convex cone generated by a 
semi-algebraic cone with the convex hull of a base of this semi-algebraic cone.
Remarkably, the results also apply to all linear images of these bases, because we 
can interpret these linear images as the bases of the dual convex cones to linear 
sections of the original hyperbolicity cone.
\par
\begin{figure}[ht!]
 \begin{tikzpicture}
    \matrix (m) [matrix of math nodes,row sep=0.1em,column
    sep=0em,minimum width=0em,
    column 1/.style={anchor=base east},
    column 2/.style={anchor=base west}] {
       e\in \cK \subset &[0em] V &[4em] \R\oplus \R^n &[4em] \R^n \\
       && \cup\\
       && C = \phi_0^{-1}(\cK)\\ 
       && C^\vee = \pi_0(\cK^\vee)\\  
       && \cap\\
       \cB=\hyper_e(\cK^\vee) \subset \cK^\vee \subset & V^\ast & 
       (\R\oplus \R^n)^\ast & (\R^n)^\ast & 
       \supset \pi(\cB)\\
    }; 
    \path[-stealth] (m-1-3) edge node
    [below]{\small $\phi_0=\pi_0^\ast$} (m-1-2); 
    \path[-stealth] (m-1-4) edge node
    [below]{\small $\phi'=\pi_2^\ast$} (m-1-3); 
    \path[-stealth] (m-6-2) edge node
    [above]{\small $\pi_0$} (m-6-3); 
    \path[-stealth] (m-6-3) edge node
    [above]{\small $\pi_2$} (m-6-4); 
    \path[-stealth, bend right] (m-1-4) edge node [below]{\small $\phi=\pi^\ast$} (m-1-2);
    \path[-stealth, bend right] (m-6-2) edge node [above]{\small $\pi$} (m-6-4);
  \end{tikzpicture}  
\caption{The point $e$ is an interior point of the convex cone $\cK$. We study the 
linear image $\pi(\cB)$ of a base $\cB$ of the dual convex cone $\cK^\vee$. The 
cone $C=\phi_0^{-1}(\cK)$ allows us to identify $\pi(\cB)$ with a base of the dual 
convex cone $C^\vee$.}
\label{eq:setup}
\end{figure}
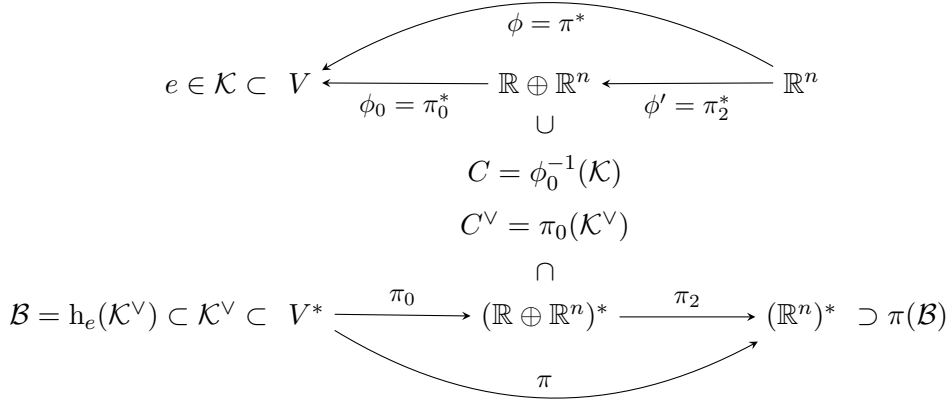
Before returning to hyperbolicity cones at the end of the section, we discuss the
necessary convex geometry. See \Cref{eq:setup} for a summary of our setup. 
\par
\begin{Def}\label{def:construction}
Let $V$ be a finite-dimensional real vector space. Let $\cK\subset V$ be a closed 
convex cone and let $e\neq 0$ be an interior point of $\cK$. The set 
\[
\cB = \hyper_e(\cK^\vee) = \{\ell\in \cK^\vee \colon \ell(e) = 1\}
\]
introduced in \Cref{eq:base} is a compact convex base of the dual convex cone 
$\cK^\vee$ by \Cref{lem:char-compact-base}. We study the image of 
$\cB$ under an arbitrary linear map 
\[
\pi \colon V^\ast \to (\R^n)^\ast.
\]
Let $\pi_0 \colon V^\ast \to (\R\oplus\R^n)^\ast$, $\ell\mapsto(\ell(e),\pi(\ell))$ 
and let $\pi_2 \colon (\R\oplus\R^n)^\ast \to (\R^n)^\ast$, $(y_0,y)\mapsto y$ denote
the projection onto the second summand. We denote the dual map to $\pi$, $\pi_0$, 
and $\pi_2$ by $\phi:\R^n\to V$, $\phi_0:\R\oplus\R^n\to V$, and 
$\phi':\R^n\to\R\oplus\R^n$, respectively. We will describe $\pi(\cB)$ in terms of 
the set
\[
C = \phi_0^{-1}(\cK)\subset\R\oplus \R^n.
\]
Let $b_0,b_1,\dots,b_n$ denote the standard basis of $\R^{n+1}=\R\oplus\R^n$.
\end{Def}
It is easy to see that the set $C$ is a closed convex cone containing $b_0$ as an 
interior point. In addition, $C$ can be seen as a linear section of the convex
cone $\cK$ through the interior point $e$. To be more precise, let 
$v_1,v_2,\ldots,v_n\in V$ and let $\pi_{v_1,\ldots,v_n}:V^\ast\to(\R^n)^\ast$ be 
defined by
\begin{equation}\label{eq:pi-explicit}
\pi_{v_1,\ldots,v_n}(\ell)=(\ell(v_1),\ldots,\ell(v_n)),
\qquad \ell\in V^\ast.
\end{equation}
If $\pi=\pi_{v_1,\ldots,v_n}$, then 
\[
\phi_0((x_0,x)^T)=x_0e+x_1v_1+\cdots+x_nv_n
\]
holds for all $(x_0,x)^T\in\R\oplus\R^n$ where $x=(x_1,\ldots,x_n)^T$. The image
of $\phi_0$ is the subspace 
$\operatorname{im}(\phi_0)=\operatorname{span}(e,v_1,\ldots,v_n)$, which intersects 
the convex cone $\cK$ in the interior point $e$. If $e,v_1,v_2,\ldots,v_n$ are 
linearly independent, then the monomorphism $\phi_0$ identifies $C$ with a linear 
section of $\cK$ through $e$.
\par
\Cref{lem:char-compact-base} shows that the set 
$\hyper_{b_0}(C^\vee)=\{(y_0,y)\in C^\vee\colon y_0=1\}$ is a (compact, convex) 
base of the dual convex cone $C^\vee$, as $b_0$ is an interior point of $C$. 
Obviously, the set $\hyper_{b_0}(C^\vee)$ is isometric under $\pi_2$ to the set
\[
\pi_2\circ\hyper_{b_0}(C^\vee)
=\{y\in(\R^n)^\ast\colon (1,y)\in C^\vee\}.
\]
We show this set is a linear image of the base $\cB$ of $\cK^\vee$.  
\par
\begin{Prop}\label{prop:pro-cone}
The image of the convex cone $\cK^\vee$ under the linear map $\pi_0$ is the closed 
convex cone $\pi_0(\cK^\vee)=C^\vee$. The image of the set $\cB$ is the compact base 
$\pi_0(\cB)=\hyper_{b_0}(C^\vee)$ of $C^\vee$. In addition we have 
$\pi(\cB)=\pi_2\circ\hyper_{b_0}(C^\vee)$.
\end{Prop}
\begin{proof}
To prove the equation $C^\vee = \pi_0(\cK^\vee)$, first let 
$y = \pi_0(\ell)\in\pi_0(\cK^\vee)$ for some $\ell\in \cK^\vee$. For all $x\in C$ 
we have $\phi_0(x)\in \cK$ and hence
\[
y(x)
=\pi_0(\ell)(x)
=\ell(\phi_0(x))
\geq 0.
\]
This shows that $y\in C^\vee$ and therefore $\pi_0(\cK^\vee)\subset C^\vee$.
We prove a partial converse by duality, i.e.~we show 
$\left(\pi_0(\cK^\vee)\right)^\vee \subset C$. Let 
$x\in \left(\pi_0(\cK^\vee)\right)^\vee$. Then 
\[
\ell(\phi_0(x))=\pi_0(\ell)(x)\geq 0
\qquad \text{for all $\ell\in \cK^\vee$,}
\]
which implies $\phi_0(x)\in \cK$ or in other words 
$x\in \phi_0^{-1}(\cK) = C$.
\par
We finish proving $C^\vee = \pi_0(\cK^\vee)$ by showing that the convex cone 
$\pi_0(\cK^\vee)$ is closed. By \Cref{lem:base-affine} it suffices to prove that 
$\pi_0(\cB)$ is a compact base of $\pi_0(\cK^\vee)$. The set $\pi_0(\cB)$ is compact 
as $\cB$ is compact. The set $\cK^\vee\setminus\{0\}$ is the cone generated by $\cB$, 
hence the set $\pi_0(\cK^\vee)\setminus\{0\}$ is the cone generated by $\pi_0(\cB)$. 
It remains to show that the origin does not lie in the affine hull of $\pi_0(\cB)$.
Since $\ell(e)=1$ holds for all $\ell\in\cB$, the affine hull of $\cB$ does not 
intersect the annihilator $e^\perp=\{\ell\in V^\ast:\ell(e)=0\}$. As 
$e\in\operatorname{im}(\phi_0)$ and $\ker(\pi_0)=\operatorname{im}(\phi_0)^\perp$,
we have $\ker\pi_0\subset e^\perp$. This shows $\aff(\cB)\cap\ker\pi_0=\emptyset$,
which proves that the origin does not lie in 
$\pi_0(\aff(\cB))=\aff(\pi_0(\cB))$.
\par
Using the equation $\pi_0(\cK^\vee)=C^\vee$ and the projection 
$\pi_1:(\R\oplus\R^n)^\ast\to\R^\ast$, $(y_0,y)\mapsto y_0$ onto the first summand, 
we get
\begin{align*}
\pi_0(\cB) &= \{\pi_0(\ell)\in(\R\oplus\R^n)^\ast \colon \ell\in\cK^\vee, \ell(e)=1\}\\
&= \{\pi_0(\ell)\in(\R\oplus\R^n)^\ast \colon \ell\in\cK^\vee, \pi_1(\pi_0(\ell))=1 \}\\
&= \{(y_0,y)\in C^\vee \colon y_0=1 \}
= \hyper_{b_0}(C^\vee).
\end{align*}
Finally, $\pi(\cB)=\pi_2\circ\pi_0(\cB)=\pi_2\circ\hyper_{b_0}(C^\vee)$ completes 
the proof.
\end{proof}
\Cref{prop:pro-cone} induces a duality of convex sets. Let 
$b_0^\ast,b_1^\ast,\dots,b_n^\ast$ denote the standard basis of 
$(\R^{n+1})^\ast=(\R\oplus\R^n)^\ast$. Let 
\mbox{$\widetilde{\pi_2}\colon \R\oplus\R^n \to \R^n$,} $(x_0,x)^T\mapsto x$ denote the 
projection onto the second summand. Clearly, the affine slice 
$\hyper_{b_0^\ast}(C)=\{(x_0,x)^T\in C\colon x_0=1\}$ of the convex cone $C$ is 
isometric under $\widetilde{\pi_2}$ to the set
\[
\widetilde{\pi_2}\circ\hyper_{b_0^\ast}(C)=\{x\in\R^n\colon (1,x)^T\in C\}.
\]
We show this set is the dual convex set to a linear image of the set $\cB$.
\par
\begin{Cor}\label{cor:convex-duality}
The dual convex set to $\pi(\cB)$ is 
$\widetilde{\pi_2}\circ\hyper_{b_0^\ast}(C)=\big(\pi(\cB)\big)^\circ$.
If the set $\pi(\cB)$ contains the origin, then 
$\pi(\cB)=\big(\widetilde{\pi_2}\circ\hyper_{b_0^\ast}(C)\big)^\circ$.
\end{Cor}
\begin{proof}
Let $x\in\R^n$. Then
\begin{align*}
x\in (\pi(\cB))^\circ 
 &\iff x\in (\pi_2\circ\hyper_{b_0}(C^\vee))^\circ\\
 &\iff 1+y(x)\geq 0            && \forall y\in\pi_2\circ\hyper_{b_0}(C^\vee)\\
 &\iff (y_0,y)(1,x)^T\geq 0    && \forall (y_0,y)\in\hyper_{b_0}(C^\vee)\\
 &\iff (1,x)^T\in C\\
 &\iff x\in\widetilde{\pi_2}\circ\hyper_{b_0^\ast}(C).
\end{align*}
The first equivalence follows from \Cref{prop:pro-cone}. Regarding the fourth 
equivalence, note that $\hyper_{b_0}(C^\vee)$ is a base of the convex cone $C^\vee$. 
The remaining equivalences follow immediately from the definitions. As $\pi(\cB)$ is a 
closed convex set, the second statement follows from the equation $S=(S^\circ)^\circ$, 
which holds for all closed convex subsets $S\subset(\R^n)^\ast$ containing the origin 
\cite{Rockafellar1970}. 
\end{proof}
The convex duality of \Cref{cor:convex-duality} has been used earlier
\cite{ChienNakazato2010,Henrion2010,HeltonSpitkovsky2012} in the context of the joint 
numerical range. The conic duality of \Cref{prop:pro-cone} has advantages in our
situation.
\par
\begin{Rem}
If the convex set $\pi(\cB)$ contains the origin, we could use the convex duality 
\begin{equation}\label{eq:conv-duality}
\pi(\cB)
=(\widetilde{\pi_2}\circ\hyper_{b_0^\ast}(C))^\circ
\end{equation}
of \Cref{cor:convex-duality} to describe the set $\pi(\cB)$. \Cref{ex:AstarnotW} and 
\Cref{fig:AstarnotW} present an example where \Cref{eq:conv-duality} fails as
$0\in\pi(\cB)$ fails. A remedy would be to translate $\pi(\cB)$ along a vector 
$(\lambda_1,\ldots,\lambda_n)\in(\R^n)^\ast$ and apply \Cref{eq:conv-duality} to the
translated set. Indeed, if $v_1,\ldots,v_n\in V$, then the map in \Cref{eq:pi-explicit} yields
\[
\pi_{v_1+\lambda_1 e,\ldots,v_n+\lambda_n e}(\cB)
=\pi_{v_1,\ldots,v_n}(\cB)+(\lambda_1,\ldots,\lambda_n).
\]
The translation is unnecessary if we use the conic duality 
\begin{equation}\label{eq:cone-duality}
\pi(\cB)
=\pi_2\circ\hyper_{b_0}(C^\vee)
\end{equation}
of \Cref{prop:pro-cone}, which is valid even if $0\in\pi(\cB)$ fails.
\par
The main reason why we employ the conic duality of \Cref{eq:cone-duality} is that 
it conveys the algebraic geometry of the dual hyperbolicity cone $C^\vee$ described 
in \Cref{sec:duality-cones} to the convex set $\pi(\cB)$ directly. The algebraic 
geometry of the hyperbolicity cone $C$ is hidden behind a convex duality if we use 
the \Cref{eq:conv-duality} to describe the set $\pi(\cB)$.
\end{Rem}
\begin{figure}[ht!]
a) 
\begin{tikzpicture}[scale=1]
\fill[color=gray!50] (-1.2,2.4) -- (0,0) -- (2.4,-0.8) -- (2.4,2.4);
\draw (-1.2,2.4) -- (0,0) -- (2.4,-0.8);
\draw[very thick] (-0.5,1) -- (2.1,1);
\draw[dotted, very thick] (2.1,1) -- (2.8,1);
\draw[very thick] (-0.5,0) -- (2.1,0);
\draw[dotted, very thick] (2.1,0) -- (2.4,0);
\draw[->] (-1.3,0) -- (2.6,0) node[right] {$x_1$};
\foreach \x/\xtext in {-1/-1, -.5/{-\tfrac{1}{2}}, 1/1, 2/2} 
\draw[shift={(\x,0)}] (0pt,2pt) -- (0pt,-2pt) node[below] {\scriptsize{$\xtext$}};
\draw[->] (0,-.9) -- (0,2.6) node[above] {$x_0$};
\foreach \y/\ytext in {1/1, 2/2} 
\draw[shift={(0,\y)}] (2pt,0pt) -- (-2pt,0pt) node[left] {\scriptsize{$\ytext$}};
\draw (0,0) node[above right] {\scriptsize{$0$}};
\draw (2.4,2.4) node[below left] {$C$};
\draw (1,1) node[above] {$\hyper_{b_0^\ast}(C)$};
\end{tikzpicture}
\hspace{3em}
b) 
\begin{tikzpicture}[scale=1]
\fill[color=gray!50] (.8,2.4) -- (0,0) -- (2.4,1.2) -- (2.4,2.4);
\draw (.8,2.4) -- (0,0) -- (2.4,1.2);
\draw[very thick] (.333,1) -- (2,1);
\draw[very thick] (.333,0) -- (2,0);
\draw[->] (-.5,0) -- (2.6,0) node[right] {$y_1$};
\foreach \x/\xtext in {.333/\tfrac{1}{3}, 1/1, 2/2} 
\draw[shift={(\x,0)}] (0pt,2pt) -- (0pt,-2pt) node[below] {\scriptsize{$\xtext$}};
\draw[->] (0,-.9) -- (0,2.6) node[above] {$y_0$};
\foreach \y/\ytext in {1/1, 2/2} 
\draw[shift={(0,\y)}] (2pt,0pt) -- (-2pt,0pt) node[left] {\scriptsize{$\ytext$}};
\draw (0,0) node[below left] {\scriptsize{$0$}};
\draw (2.4,2.4) node[below left] {$C^\vee$};
\draw (1.3,1) node[above] {$\hyper_{b_0}(C^\vee)$};
\end{tikzpicture}
\caption{%
a) The projection of the affine slice $\hyper_{b_0^\ast}(C)$ of the convex cone $C$ 
to the $x_1$-axis is the interval $[-\tfrac{1}{2},\infty)=[\tfrac{1}{3},2]^\circ$.
b) The projection of the base $\hyper_{b_0}(C^\vee)$ of the dual convex cone 
$C^\vee$ to the $y_1$-axis is the interval 
$\pi(\cB)=[\tfrac{1}{3},2]\neq[-\tfrac{1}{2},\infty)^\circ$.}
\label{fig:AstarnotW}
\end{figure}
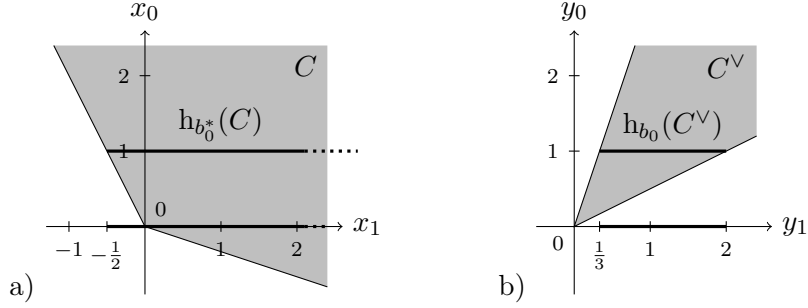
\begin{Exm}\label{ex:AstarnotW}
Let $\cK=\{(\xi_1,\xi_2)^T\in\R^2:\xi_1,\xi_2\geq 0\}$ be the nonnegative quadrant 
of the plane $V=\R^2$ and let $e=(1,1)^T\in V$. The segment 
\[
\cB
=\hyper_e(\cK^\vee)
=\big[(0,1),(1,0)\big]
\]
is a base of the dual convex cone $\cK^\vee=\{z^T:z\in\cK\}$. Let 
$v=(2,\tfrac{1}{3})^T\in V$. The image of $\cB$ under the map $\pi:V^\ast\to\R^\ast$, 
$z\mapsto z(v)$ is the interval $[\tfrac{1}{3},2]$, which does not contain the origin, 
$0\not\in\pi(\cB)$. We visualize \Cref{cor:convex-duality}, \Cref{eq:cone-duality}, 
and the failure of \Cref{eq:conv-duality} in \Cref{fig:AstarnotW}.
\end{Exm}
We formulate the results from \Cref{sec:duality-cones} in terms of bases of cones,
taking the convex cone $\cK\subset V$ in \Cref{def:construction} equal to a 
hyperbolicity cone. 
\par
\begin{Thm}\label{thm:piB}
Let $p$ be a hyperbolic polynomial on $V$ with hyperbolicity direction $e\in V$ 
and hyperbolicity cone $C_e(p)\subset V$. The image of the compact convex base 
$\hyper_e(C_e(p)^\vee)$ of the dual convex cone $C_e(p)^\vee$ under the linear map $\pi$ is 
\[
\pi\Big(\hyper_e\big(C_e(p)^\vee\big)\Big)
=\cv\big(\cl(T_1\cup T_2\cup\dots\cup T_r)\big).
\]
Here, $p_1^{m_1}p_2^{m_2}\dots p_r^{m_r}$ is a factorization of the 
pullback $p\circ\phi_0\in\R[x_0,x_1,\dots,x_n]$ into irreducible polynomials, and
\[
T_i = \{ (y_1,\dots,y_n) \in (\R^n)^\ast \mid
(1:y_1:\dots:y_n)\in (\cV(p_i)^\ast)_\reg \}, 
\quad i=1,\dots,r.
\]
\end{Thm}
\begin{proof}
The pullback $p\circ\phi_0$ is a hyperbolic polynomial on $\R\oplus\R^n$ with 
hyperbolicity direction $b_0$. As 
$\phi_0((x_0,x)^T)=x_0e+\phi(x)$ holds for all points $(x_0,x)^T\in\R\oplus\R^n$, 
we have $p\circ\phi_0(b_0)=p(e)>0$. The equation
\begin{equation}\label{eq:pull-back}
p\circ\phi_0(tb_0-a)
=p(te-\phi_0(a))
\end{equation}
shows that the polynomial $p\circ\phi_0(tb_0-a)$ in one variable $t$ has only real 
roots for every point $a\in\R\oplus\R^n$ so that $p\circ\phi_0$ is indeed hyperbolic 
with respect to the point $b_0$. \Cref{eq:pull-back} also shows 
\[
C_{b_0}(p\circ\phi_0)
=\phi_0^{-1}\left(C_e(p)\right).
\]
Using the hyperbolicity cone $\cK=C_e(p)$ in \Cref{prop:pro-cone}, we get
\[
\pi\Big(\hyper_e\big(C_e(p)^\vee\big)\Big)
=\pi_2\circ\hyper_{b_0}\big(C_{b_0}(p\circ\phi_0)^\vee\big).
\]
\par
If $p_1^{m_1}p_2^{m_2}\dots p_r^{m_r}=p\circ\phi_0$ is a factorization into 
irreducible polynomials, then
\[
C_{b_0}(p\circ\phi_0)
=C_{b_0}(p_1)\cap C_{b_0}(p_2)\cap\dots\cap C_{b_0}(p_r)
\]
holds. \Cref{thm:dual-hyp-cone} proves $C_{b_0}(p_i)^\vee=\cl(\co(C_i))$, 
$i=1,\dots,r$, where
\[
C_i
=\big\{ (y_0,y_1,\dots,y_n) \in (\R\oplus\R^n)^\ast \mid
(y_0:y_1:\dots:y_n)\in (\cV(p_i)^\ast)_\reg, y_0\geq 0 \big\}.
\]
Now \Cref{eq:inter-dual-cl} in \Cref{pro:dual-of-intersection} yields
\[
\hyper_{b_0}\Big(C_{b_0}(p\circ\phi_0)^\vee\Big)
= \cv\Big(\cl\big(\hyper_{b_0}(C_1)\cup\hyper_{b_0}(C_2)
\cup\dots\cup\hyper_{b_0}(C_r)\big)\Big).
\]
The claim follows since the map $\pi_2:(\R\oplus\R^n)^\ast\to(\R^n)^\ast$ restricts 
to an affine isomorphism from $\hyper_{b_0}((\R\oplus\R^n)^\ast)$ onto $(\R^n)^\ast$, 
and since $\pi_2\circ\hyper_{b_0}(C_i)=T_i$ holds for $i=1,\dots,r$.
\end{proof}
Like in \Cref{cor:dual-hyp-cone-bound}, some of the factors of the pullback
$p\circ\phi_0=p_1^{m_1}p_2^{m_2}\dots p_r^{m_r}$ are redundant in \Cref{thm:piB}.
Let $I\subset\{1,2,\dots,r\}$ be the subset such that $i\in I$ if and only if
the hypersurface $\cV(p_i)$ belongs to the algebraic boundary
$\partial_a C_{b_0}(p\circ\phi_0)$ of the hyperbolicity cone 
$C_{b_0}(p\circ\phi_0)$ for $i=1,\dots,r$. Then 
$\pi(\cB)=\cv\Big(\cl\big(\bigcup_{i\in I}T_i\big)\Big)$.
\par
%
%
\section{The Cone of Positive-Semidefinite Matrices}
\label{sec:PSD-cone}
Here we prove the theorems stated in the introduction by applying our results 
to the hyperbolicity cone of positive semi-definite matrices. We discuss
block diagonalization \emph{versus} factorization of the determinant. We present
examples, among them the promised example by Chien and Nakazato.
\par
Let $V=H_d$ be the space of Hermitian $d\times d$ matrices. The determinant is a 
hyperbolic polynomial on $V$ with respect to any positive definite matrix. The 
hyperbolicity cone is the set of positive semi-definite matrices 
\[
\cK
=C_\id(\det)
=\{A\in H_d:A\succeq 0\}.
\]
This convex cone is a self-dual convex cone, i.e.~$\cK^\vee = \cK$, as we identify 
the dual space $V^\ast$ with $V$ using the scalar product 
$\langle A,B\rangle=\tr(AB)$, $A,B\in H_d$. In the notation of 
\Cref{def:construction}, the base $\hyper_\id(\cK^\vee)$ of $\cK^\vee$ is the 
set of density matrices introduced in \Cref{eq:density-matrices},
\[
\cB
=\hyper_\id(\cK^\vee)
=\{\rho\in H_d \colon \rho\succeq 0, \tr(\rho)=1\}.
\]
Let $A_1,\dots,A_n\in H_d$ be Hermitian matrices 
and let $\pi:H_d\to(\R^n)^\ast$ be the linear map defined 
by $\pi(A)=\langle A,A_i\rangle_{i=1}^n$ for all $A\in H_d$. The image of the set 
$\cB$ under $\pi$ is the joint numerical range defined in \Cref{eq:jnr},
\[
W=\pi(\cB)=\{\langle \rho,A_i\rangle_{i=1}^n:\rho\in\cB\}.
\]
The dual map $\phi=\pi^\ast$ has the values $\phi(x)=x_1 A_1+\cdots+x_n A_n$ and the 
map $\phi_0:\R\oplus\R^n\to V$, see \Cref{def:construction}, has the values 
$\phi_0[(x_0,x)^T]=x_0\id+\phi(x)$ for all $x_0\in\R$ and $x=(x_1,\dots,x_n)^T\in\R^n$. 
Hence, the pullback $\det\circ\phi_0$ takes the form
\begin{equation}\label{eq:p-detphi}
p
=\det\circ\phi_0
=\det(x_0\id+x_1 A_1+\cdots+x_n A_n)
\in\R[x_0,x_1,\dots,x_n].
\end{equation}
With this preparation, \Cref{eq:modified-K} follows from \Cref{thm:piB} immediately.
To obtain Kippenhahn's original result, as stated in \Cref{thm:k}, we use 
\Cref{cor:dim3} instead of \Cref{cor:dual-hyp-cone} in the proof of \Cref{thm:piB}.
\par
\begin{Rem}
In the matrix setting of this section, the convex cone $\phi_0^{-1}(\cK)$ is the 
\emph{spectrahedral cone} \cite{Netzer2012}
\[
\{x\in\R^{n+1} : x_0\id+x_1 A_1+\cdots+x_n A_n \succeq 0 \},
\]
which is the hyperbolicity cone $C_{b_0}(p)$ of the polynomial in \Cref{eq:p-detphi}, 
as we observed more generally in the proof of \Cref{thm:piB} above.
\end{Rem}
Researchers have studied the relationships between the possibility to block 
diagonalize the Hermitian matrices $A_1,\ldots,A_n\in H_d$ simultaneously and to
factor the determinant $p$ in \Cref{eq:p-detphi}. The matrices $A_1,\ldots,A_n$ 
are called {\em unitarily reducible} if there is a $d$-by-$d$ unitary $U$ and an 
integer $0<j<d$ such that $UA_iU^\ast=A_{i,1}\oplus A_{i,2}$ is a block diagonal 
matrix, where $A_{i,1}\in H_j$ and $A_{i,2}\in H_{d-j}$ for all $i=1,\dots,n$. 
Otherwise the matrices $A_1,\ldots,A_n$ are {\em unitarily irreducible}. 
\par
Clearly, the polynomial $p\in\R[x_0,x_1,x_2]$ is the product of the polynomials 
corresponding to the two diagonal blocks for every pair of unitarily reducible 
matrices $A_1,A_2\in H_d$. Interestingly, $p$ may even factor when  the matrices 
$A_1,A_2$ are unitarily irreducible.
\par
\begin{Rem}[Kippenhahn's conjecture]\label{ex:Laffey}
Kippenhahn \cite[Prop.~28a]{Kippenhahn1951} conjectured that if $p$ has a repeated 
factor, then the matrices $A_1,A_2$ are unitarily reducible. Laffey \cite{Laffey1983} 
found a first counterexample to the conjecture by presenting a pair of unitarily 
irreducible hermitian $8$-by-$8$ matrices $A_1,A_2\in H_8$, where $p$ is the square 
of a polynomial of degree four. The paper \cite{Buckley2016} summarizes more recent
work on the topic and shows that every pair of hermitian $6$-by-$6$ matrices 
$A_1,A_2\in H_6$ is unitarily reducible if $p$ is the square of a polynomial of 
degree three that defines a smooth cubic.
\end{Rem}
\par
Similarly, one may ask whether the polynomial $p\in\R[x_0,x_1,x_2]$ can be a product 
of several irreducible factors of multiplicity one if the hermitian matrices 
$A_1,A_2\in H_d$ are unitarily irreducible. Such questions have been studied in detail by Kerner and Vinnikov (\cite{Kerner-Vinnikov2011,Kerner-Vinnikov2012}), resulting in a number of necessary and sufficient conditions. For example, the union of three lines in the projective plane admits a unitarily irreducible determinantal representation (see \cite[Remark 3.6]{Kerner-Vinnikov2011}).

\medskip
Here is a simple example demonstrating that the dual varieties to the irreducible 
components of the hypersurface $\cV(p)$ may have different dimensions.
\par
\begin{Exm}\label{ex:drop}
Let
\[
A_1=\left(
\begin{array}{ccc} 
    0 &    1 &    0 \\
    1 &    0 &    0 \\
    0 &    0 &    2
\end{array}\right),
\qquad
A_2=\left(
\begin{array}{ccc} 
    0 & -\ii &    0 \\
  \ii &    0 &    0 \\
    0 &    0 &    0
\end{array}\right),
\qquad
A_3=\left(
\begin{array}{ccc} 
    1 &    0 &    0 \\
    0 &   -1 &    0 \\
    0 &    0 &    0
\end{array}\right).
\]
The irreducible components of $\cV(p)$ are 
$X_1=\{x\in\P^3:x_1^2+x_2^2+x_3^2=x_0^2\}$ and $X_2=\{x\in\P^3:x_0+2 x_1=0\}$. 
The dual varieties are $X_1^\ast=\{y\in(\P^3)^\ast:y_1^2+y_2^2+y_3^2=y_0^2\}$,
by \Cref{exm:quadric}, and $X_2^\ast=\{(1:2:0:0)\}$. Both $X_1^\ast$ and 
$X_2^\ast$ are smooth, so $T_1=\{y\in(\R^3)^\ast:y_1^2+y_2^2+y_3^2=1\}$ and 
$T_2=\{(2,0,0)\}$. The joint numerical range $W=\cv(T_1\cup T_2)$ is the convex 
hull of a sphere and a point outside the sphere.
\end{Exm}
\begin{figure}[ht!]%
a)\includegraphics[height=3.6cm]{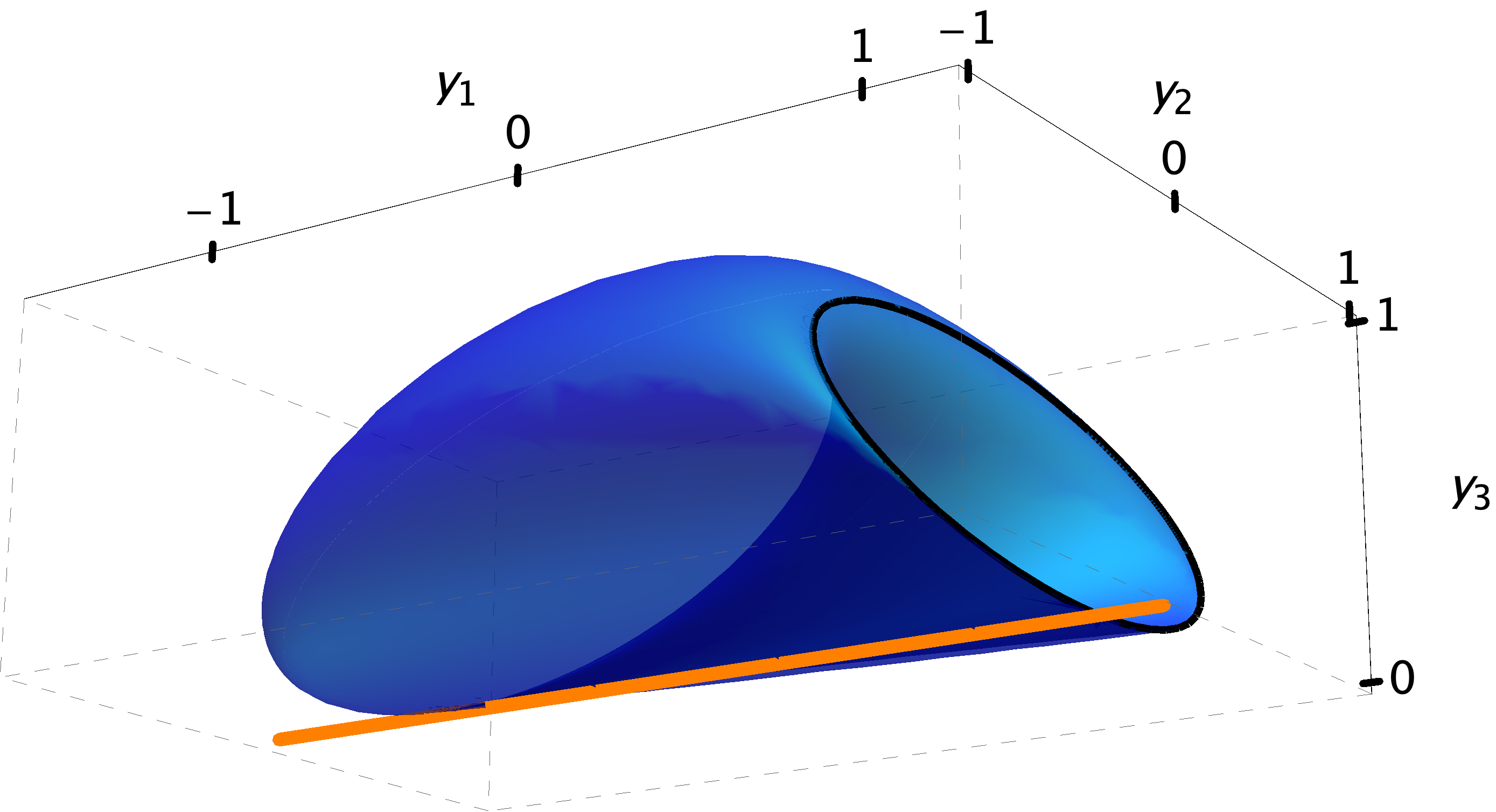}\\%
b)\includegraphics[height=3.6cm]{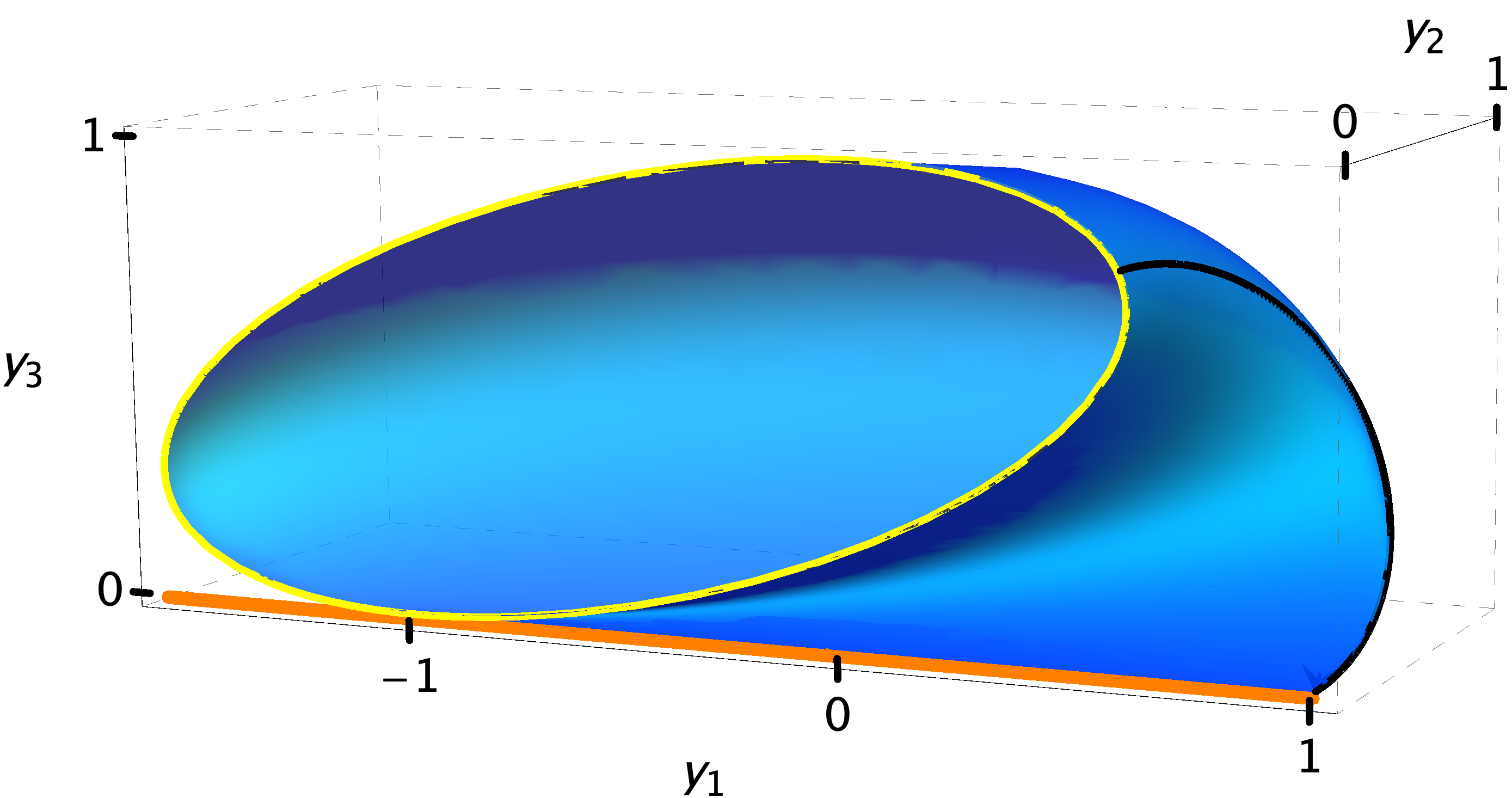}%
\caption{%
Pictures for \Cref{ex:ChienNakazato}: a) 
Surface $T$ (blue), 
singular locus $y_1$-axis (orange), 
and ellipse on the boundary of the joint numerical range (black). 
b) 
The hyperplane $y_2=0$ intersects $T$ in the union of an ellipse (yellow) 
and the $y_1$-axis (orange).}
\label{fig:ChienNakazato}
\end{figure}
In our last example, we discuss the original example given by Chien and Nakazato 
from the point of view of real algebraic geometry in detail. The example 
explains why singular points of the dual varieties $\cV(p_i)^\ast$ have to be 
excluded from the statement of \Cref{thm:piB}. 
\par
\begin{Exm}[Chien and Nakazato \cite{ChienNakazato2010}]\label{ex:ChienNakazato}
Let
\[
A_1=\left(
\begin{array}{ccc} 
    1 &    0 &    0 \\
    0 &   -1 &    1 \\
    0 &    1 &    0
\end{array}\right),
\qquad
A_2=\left(
\begin{array}{ccc} 
    0 &    0 & -\ii \\
    0 &    0 &    0 \\
  \ii &    0 &    0
\end{array}\right),
\qquad
A_3=\left(
\begin{array}{ccc} 
    0 &    0 &    0 \\
    0 &    0 &    0 \\
    0 &    0 &    1
\end{array}\right).
\]
The hyperbolic cubic form
\begin{align*}
p
 &= \det( x_0 \id + x_1 A_1 + x_2 A_2 + x_3 A_3 )\\
 &= x_0^3 + x_0^2x_3 - 2x_0x_1^2 - x_0x_2^2 - x_1^3 - x_1^2x_3 + x_1x_2^2
\end{align*}
is irreducible. The dual variety to $X=\{x\in\P^3: p(x)=0\}$ is the hypersurface 
$X^\ast=\{y\in(\P^3)^\ast: q(y)=0\}$ defined by the homogeneous quartic form
\begin{align*}
q
= & \, 4 y_0^2 y_3^2
+ 8 y_0 y_1 y_3^2
- 4 y_0 y_2^2 y_3
- 24 y_0 y_3^3
+ 4 y_1^2 y_3^2
- 4 y_1 y_2^2 y_3
- 8 y_1 y_3^3\\
& \, + y_2^4
+ 8 y_2^2 y_3^2
+ 20 y_3^4.
\end{align*}
It is easy to show that the singular locus $X^\ast\setminus(X^\ast)_\reg$ of $X^*$ 
is the line $\{(y_0,y_1,y_2,y_3)\in(\P^3)^\ast: y_2=y_3=0\}$. 
\par
The surface $T=\{(y_1,y_2,y_3)\in(\R^3)^\ast \mid (1:y_1:y_2:y_3)\in X^\ast\}$, 
depicted in \Cref{fig:ChienNakazato}, is the zero-set of the polynomial 
$Q(y_1,y_2,y_3)=q(1,y_1,y_2,y_3)$. All points $(y_1,0,0)\in T$ on the line of 
singular points $y_2=y_3=0$ in the interval $|y_1|\leq 1$ are central points of 
$X^\ast$. This can be seen from a parametrization of $X$ described in 
\cite{SchwonnekWerner2018}. 
In addition, the points $(y_1,0,0)\in T$ with $|y_1|>1$ are not central. This can
be seen from the roots of the quadratic polynomial 
$R_{y_1,y_3}\in\R[y_2]$ satisfying $R_{y_1,y_3}(y_2^2)=Q(y_1,y_2,y_3)$. In fact,
if $(y_1,y_3)\in(\R^2)^\ast$, then the line 
$\{(y_1,\lambda,y_3) \colon \lambda\in\R\}$ intersects $T$ if and only if 
$R_{y_1,y_3}$ has a non-negative root $y_2$. It is not hard to see (for example 
using Sturm's theorem \cite{BCR}) that, under the assumption $y_3\neq 0$, the 
polynomial $R_{y_1,y_3}$ has a non-negative root if and only if $(y_1,y_3)$ lies 
in the convex hull of the union of the singleton $\{(1,0,0)\}$ with the ellipse 
given by $y_1^2 + 5 y_3^2 - 2 y_1 y_3 + 2 y_1 - 6 y_3  + 1 = y_2 = 0$, the yellow 
curve in \Cref{fig:ChienNakazato}~b). Therefore, the points $(y_1,0,0)$ with 
$|y_1|>1$ are not central. These are the points which are removed from $X^\ast(\R)$ 
in the statement of \Cref{eq:modified-K}. The convex hull of the remainder 
of $T$ is the joint numerical range $W$.
\par
For a general, regular point $(y_0:y_1:y_2:y_3)$ on $X^*$, the dual hyperplane 
in $\P^3$ defined by $x_0y_0 + x_1y_1 + x_2y_2 + x_3y_3 = 0$ intersects $X$ 
in an irreducible but singular cubic. The hyperplane sections of 
the hypersurface $X$ corresponding to points of the line of singular points of 
$X^\ast$ are special in the following way: The sections are reducible curves in 
$\P^2$ that factor into a conic and a line, which is tangent to the conic at a 
real point. Not all of the points on this line are central points of $X^*$. The 
central points of $X^*$ on this line are exactly those that we can perturb such 
that the dual hyperplane section of $X$ deforms to a real cubic with a real 
singularity. The other points on this line can only be perturbed on $X^*$ to 
complex points, which means that the dual hyperplane section of $X$ only deforms 
to complex singular hyperplane sections that are complex cubics with complex 
singularities.
\par
To see this explicitly, we compute the restrictions of $p$ to hyperplanes in $\P^3$ 
defined by $x_0 y_0 + x_1 y_1 = 0$. Let us assume for simplicity for now that 
$y_0 \neq 0$ so that $x_0 = a x_1$ for some $a\in \R$. Then $p$ factors
\[
p(a x_1,x_1,x_2,x_3) = 
 x_1 \left( (a^3 - 2a - 1) x_1^2 + (1 - a) x_2^2 + (a^2 - 1) x_1 x_3 \right).
\]
So in the plane defined by $x_0y_0 + x_1y_1=0$, we see the conic 
\[
(a^3 - 2a - 1) x_1^2 + (1 - a) x_2^2 + (a^2 - 1) x_1 x_3 = 0
\] 
and the line defined by $x_0 = x_1 = 0$, which is tangent to the conic at the point 
$(0:0:0:1)$, which is a singular point of $X$. (The computation for the case $y_0 = 0$ 
is similar.) For every real value of the parameter $a$, the conic is real and 
indefinite of full rank except for $a=-1$ and $a=1$. These two values of $a$ bound 
the interval of central points on the line of singular points on $X^\ast$; see 
\Cref{fig:ChienNakazato}.
\end{Exm}
Section~6 of \cite{Szymanski-etal2018} presents further examples 
of surfaces analogous to the surface $T$ in \Cref{ex:ChienNakazato} that contain 
straight lines. These lines are identified in the paper 
\texttt{arXiv:1603.06569v3 [math.FA]} on the preprint server arXiv.
\par
So why do naive generalizations of Kippenhahn's Theorem fail in higher dimensions? 
The reason is that hyperplanes and lines are only the same in $\P^2$. A central 
point of the argument in the proof of Kippenhahn's Theorem is \Cref{lem:multiplicity}, 
which argues that there are no lines through the interior of the hyperbolicity cone 
that are tangent to the hyperbolic hypersurface (\Cref{Cor:multiplicity}). 
The hyperplane corresponding to a real point of the dual variety may well meet an 
interior point of the hyperbolicity cone as long as the hyperplane contains no tangent 
line to the hyperbolic hypersurface incident with this interior point. Our 
generalization holds essentially because this can only happen for hyperplanes that are 
not central points of the dual variety.
\par
\subsection*{Acknowledgements} 
We thank H.~Nakazato for helpful comments.
SW thanks K.~{\.Z}yczkowski for discussions of numerical ranges in quantum 
mechanics. 
We would like to thank the anonymous referees for their suggestions to improve 
the paper.
\par
%
%
%
\bibliographystyle{plain}

%
%
%
\vspace{\fill}
\parbox{8cm}{%
Daniel Plaumann\\
Technische Universit{\"a}t Dortmund\\
Fakult{\"a}t f{\"u}r Mathematik\\
Vogelpothsweg 87\\
44227 Dortmund\\
Germany\\
e-mail \texttt{Daniel.Plaumann@tu-dortmund.de}}
\vspace{\baselineskip}
\par\noindent
\parbox{8cm}{%
Rainer Sinn\\
Freie Universit{\"a}t Berlin\\
Fachbereich Mathematik und Informatik\\
Arnimallee 2\\
14195 Berlin\\
Germany\\
e-mail \texttt{rainer.sinn@fu-berlin.de}}
\vspace{\baselineskip}
\par\noindent
\parbox{10cm}{%
Stephan Weis\\
Centro de Matem{\'a}tica da Universidade de Coimbra\\
Apartado 3008\\
EC Santa Cruz\\
3001 - 501 Coimbra\\
Portugal\\
e-mail \texttt{maths@weis-stephan.de}}
\end{document}